\documentclass[11pt,letterpaper]{amsart}
\usepackage{microtype}
\usepackage[dvipsnames]{xcolor}
\usepackage[croatian,english]{babel}
\usepackage{colonequals}
\usepackage{appendix}
\usepackage{graphicx}
\usepackage{amsmath,amssymb}
\usepackage{varwidth}
\usepackage{versions}
\usepackage{tikz-cd}
\usepackage{tikz}
\usetikzlibrary{calc}
\usepackage{mathrsfs}
\usepackage{enumitem}
\usepackage[colorlinks=true,hyperindex,pagebackref,linktocpage=true]{hyperref}
\usepackage{cleveref}

\hypersetup{
	colorlinks,
	linkcolor={teal!60!blue},
	citecolor={Goldenrod!60!Maroon},
	urlcolor={purple!80!black}
}

\newtheorem{thm}{Theorem}[section]
\newtheorem{conj}[thm]{Conjecture}
\newtheorem{prop}[thm]{Proposition}
\newtheorem{lem}[thm]{Lemma}
\newtheorem{defn}[thm]{Definition}
\newtheorem{cor}[thm]{Corollary}
\theoremstyle{remark}
\newtheorem{rem}[thm]{Remark}
\newtheoremstyle{named}{}{}{\itshape}{}{\bfseries}{.}{.5em}{#1 \thmnote{#3}}
\theoremstyle{named}
\newtheorem{namedtheorem}{GHN}
\newcommand{\GHN}[2]{%
  \begin{namedtheorem}[#1]\label{#2}
}

\crefname{lem}{lemma}{lemmas}
\Crefname{lem}{Lemma}{Lemmas}
\crefname{thm}{theorem}{theorems}
\Crefname{thm}{Theorem}{Theorems}
\crefname{defn}{definition}{definitions}
\Crefname{defn}{Definition}{Definitions}
\crefname{prop}{proposition}{propositions}
\Crefname{prop}{Proposition}{Propositions}
\crefname{cor}{corollary}{corollaries}
\Crefname{cor}{Corollary}{Corollaries}
\crefname{conj}{conjecture}{conjectures}
\Crefname{conj}{Conjecture}{Conjectures}
\crefname{namedtheorem}{ghn}{ghns}
\Crefname{namedtheorem}{GHN}{GHNs}

\newcommand{\extw}{\widetilde{W}}

\newcommand{\brF}{\breve{F}}
\newcommand{\brG}{\breve{G}}
\newcommand{\bbQ}{\mathbb{Q}}
\newcommand{\bbS}{\mathbb{S}}
\newcommand{\calO}{{\mathcal{O}}}
\newcommand{\etasigma}{\eta_\sigma}

\DeclareMathOperator{\defe}{def}
\DeclareMathOperator{\nr}{nr}

\DeclareMathOperator{\supp}{supp}
\DeclareMathOperator{\id}{id}
\DeclareMathOperator{\Adm}{Adm}

\DeclareMathOperator{\rk}{rk}

\DeclareMathOperator{\GL}{GL}
\DeclareMathOperator{\PGL}{PGL}
\newcommand{\sigmasupp}{\supp_\sigma}
    \title{Nonemptiness of single affine Deligne-Lusztig varieties}
	
\begin{document}

	\author[D.G. Lim]{Dong Gyu Lim}
	
	\address{Evans Hall, University of California at Berkeley, CA, USA}
	\email{limath@math.berkeley.edu}
	
	\begin{abstract}
		Affine Deligne-Lusztig varieties with various level structures show up in the study of Shimura varieties and moduli spaces of shtukas. Among is the Iwahori level structure which is the most refined one. We study the nonemptiness problem of single affine Deligne-Lusztig varieties at Iwahori level in the basic case. Under a genericity condition (the ``shrunken Weyl chambers'' condition), an explicit criterion is known. However, no explicit criterion has been available without the condition even conjecturally. We conjecture a new criterion in full generality, and prove it except for finitely many cases. As an application, the nonemptiness problem for special cases and a new conjectural dimension formula are discussed.
	\end{abstract}
	\date{\today}
	
	\maketitle
	\tableofcontents

\section{Introduction}

\subsection{Background}

In the study of the special fibers of Shimura varieties, affine Deligne-Lusztig varieties show up naturally. In his seminal expository article \cite{Ra05}, Rapoport introduced \textit{affine Deligne-Lusztig variety} over a mixed characteristic local field as an important piece in the description of $\overline{\mathbb{F}}_p$-points of the special fiber of a certain Shimura variety with hyperspecial level structure or Iwahori level structure. This was motivated by the Langlands-Rapoport conjecture in that the $p$-part of the conjecture is the affine Deligne-Lusztig variety. Since then, affine Deligne-Lusztig varieties have been exploited in the study of Shimura varieties, Rapoport-Zink spaces, local Shimura varieties, and moduli spaces of local shtukas.

Many questions arise naturally including the geometric (or scheme) structure, the nonemptiness, the dimension formula, and the set of connected components, etc. These basic questions are not only interesting on their own but useful in the study of the aforementioned objects. For example, the set of connected components computed in \cite{CKV} is used to prove the Langlands-Rapoport conjecture in \cite{Kis17}. It (resp. the dimension formula) can also be used to describe the set of connected components (resp. the dimension) of Rapoport-Zink spaces (\cite[3]{Shen}, \cite[3.2]{Zhu}). In this paper, we focus on the nonemptiness question and the dimension formula.
\subsubsection{Mazur's Inequality}\label{mazurineq}
Let $G$ be a connected reductive group over $\mathbb{Z}_p$ with a maximal torus $T$ and let $\breve{\mathbb{Q}}_p$ be the fraction field of the ring of Witt vectors $W(\overline{\mathbb{F}}_p)\equalscolon\breve{\mathbb{Z}}_p$. The Frobenius morphism on $\overline{\mathbb{F}}_p$ lifts uniquely to $\breve{\mathbb{Z}}_p$ by the universal property of the ring of Witt vectors and then extend to a bijective map (denoted by $\sigma$) on $\breve{\bbQ}_p$. Now, fix $b\in G(\breve{\bbQ}_p)$ and a dominant cocharacter $\mu\in X_*(T)^+$. The affine Deligne-Lusztig variety is defined as \[X_\mu(b)\colonequals\{gG(\breve{\mathbb{Z}}_p)\in G(\breve{\bbQ}_p)/G(\breve{\mathbb{Z}}_p):g^{-1}b\sigma(g)\in G(\breve{\mathbb{Z}}_p)p^\mu G(\breve{\mathbb{Z}}_p)\},\]where $p^\mu$ is the image of $p$ under the cocharacter $\mu$.

The very first result on the nonemptiness is due to Rapoport-Richartz (\cite[Theorem 4.2]{RR96}). They showed that, when $G$ is unramified, if $X_\mu(b)$ is nonempty then \textit{Mazur's inequality}, that is, $[b]\in B(G,\mu)$ (see \Cref{bgmu}) holds. Thanks to \cite{Kottwitz03}, \cite{Gashi}, and \cite{He14}, it is now a theorem that $X_\mu(b)\neq\emptyset$ if and only if $[b]\in B(G,\mu)$ for a general reductive group $G$.

Similarly, for $G$ a connected reductive group over $\mathbb{Q}_p$, one can consider affine Deligne-Lusztig varieties with an arbitrary parahoric level structure $K$ as follows using Bruhat-Tits theory.\begin{align*}X(\mu,b)_{K}\colonequals\{gK\in G(\breve{\bbQ}_p)/K:g^{-1}b&\sigma(g)\in K \dot{x} K\\&\text{ for some }x\in W_{K}\backslash\Adm(\mu)/W_{K}\},\end{align*}where $W_K$ is the group generated by the simple reflections of $K$ and $\Adm(\mu)$ is the $\mu$-admissible set (\cite{KR00}). Still, Mazur's inequality is the necessary and sufficient condition for $X(\mu,b)_K\neq\emptyset$ (see \cite{KoRa03}, \cite{Win05}, and \cite{He16}).

Meanwhile, $X(\mu,b)_K$ is, from the definition, a disjoint union of several pieces (where $x$ varies over $W_K\backslash\!\Adm(\mu)/W_K$). These pieces, therefore, can be thought of as more refined objects or building blocks, which we call as \textit{single}\footnote{This is the terminology used in the literature occasionally. See \cite{He14} for example.} affine Deligne-Lusztig varieties. One can study their nonemptiness problem and, to get to the point first, it has a considerably different flavor.

We remark that, among parahoric level structures, the Iwahori level structure contains the finest information via the natural projection map from the affine flag variety (Iwahori level) to the affine partial flag variety (parahoric level) or the affine Grassmannian (hyperspecial level). From now on, we restrict ourselves to the Iwahori level.

\subsection{Single affine Deligne-Lusztig variety at Iwahori level}\label{reumanconj}

Along with the notations from \Cref{mazurineq}, let $I$ be a $\sigma$-stable Iwahori subgroup of $G(\breve{\bbQ}_p)$ stabilizing a base alcove, $\extw$ be the Iwahori-Weyl group, and $W_0$ be the relative Weyl group (see \Cref{IwahoriWeyl}). For $x\in \extw$ and $b\in G(\breve{\bbQ}_p)$, the single affine Deligne-Lusztig variety (at Iwahori level) is defined by\[X_x(b)\colonequals\{gI\in G(\breve{\bbQ}_p)/I:g^{-1}b\sigma(g)\in I\dot{x}I\}\]where $\dot{x}$ is an element in a subgroup of $G(\breve{\bbQ}_p)$ which is a lift of $x\in \extw$.

The question on the nonemptiness criterion (and the dimension formula) for $X_x(b)$ first appeared in \cite[Problem 4.5]{Ra00} and some cases when $G=\GL_2$ were studied in \cite[Example 4.3]{Ra05}. For a more general setting, the first partial conjecture was posed by Reuman (\cite[Conjecture 7.1]{Reuman}) where he detected an element in $W_0$, now called $\eta_\sigma(x)$ (see \Cref{etasigma}), that gives a good amount of information on the nonemptiness (and even on the dimension formula). The conjecture was reformulated and proved partially in \cite{GHKR} and then completely by \cite[Theorem B]{GHN} in the basic case. For convenience, let us name the theorem \Cref{ghnb}.

We note first that there is an obvious obstruction for the nonemptiness. Recall the Kottwitz map (cf. \Cref{bgmu}) defined in \cite[4.5]{Kottwitz}. Then,\begin{center}if $X_x(b)\neq\emptyset$ then $\kappa_G(x)=\kappa_G(b)$,\end{center} because the Kottwitz map applied to the condition $g^{-1}b\sigma(g)\!\in\! I\dot{x}I$ results in $\kappa_G(b)\!=\!\kappa_G(\dot{x})\!=\!\kappa_G(x)$. Geometrically, this means simply `$x$ and $b$ are in the same connected component of the loop group $G(\breve{\bbQ}_p)$'.

We may and will reduce to the case where $G$ is simple, quasi-split, and adjoint (\Cref{sadlv}). Then, \Cref{ghnb} tells us that, under the condition called \textit{Shrunken Weyl chambers} (\Cref{shrunken}), there is only one interesting obstruction ($\sigmasupp(\etasigma(x))=\bbS$) other than the obvious one ($\kappa_G(x)=\kappa_G(b)$).

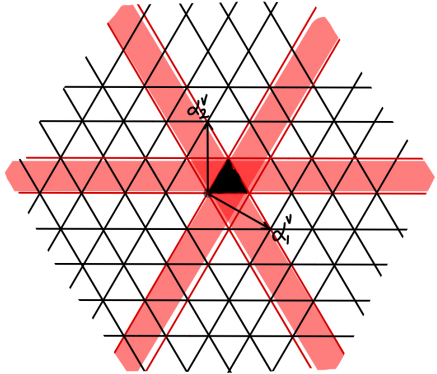
\begin{figure}[htbp]
\centering
\begin{tikzpicture}[scale=0.75]

\def\N{11}
\def\H{4.75}

\coordinate (v1) at (1,0);
\coordinate (v2) at (0.5,{sqrt(3)/2});
\coordinate (v3) at (-0.5,{sqrt(3)/2});

\begin{scope}
\clip
  ( \H,0)
  -- ({\H/2},{\H*sqrt(3)/2})
  -- ({-\H/2},{\H*sqrt(3)/2})
  -- (-\H,0)
  -- ({-\H/2},{-\H*sqrt(3)/2})
  -- ({\H/2},{-\H*sqrt(3)/2})
  -- cycle;


\fill[red!45]
  ($ -6*(v2) $) -- ($ 6*(v2) $) -- 
  ($ 6*(v2)+(v1) $) -- ($ -6*(v2)+(v1) $) -- cycle;

\fill[red!45]
  ($ -6*(v3) $) -- ($ 6*(v3) $) -- 
  ($ 6*(v3)+(v2) $) -- ($ -6*(v3)+(v2) $) -- cycle;

\fill[red!45]
  ($ -6*(v1) $) -- ($ 6*(v1) $) -- 
  ($ 6*(v1)+(v3) $) -- ($ -6*(v1)+(v3) $) -- cycle;


\foreach \k in {-\N,...,\N}{
  \draw[black!66,line width=0.1pt]
    ($ -\N*(v1) + \k*(v2) $) --
    ($  \N*(v1) + \k*(v2) $);
}
\foreach \k in {-\N,...,\N}{
  \draw[black!66,line width=0.1pt]
    ($ -\N*(v2) + \k*(v1) $) --
    ($  \N*(v2) + \k*(v1) $);
}
\foreach \k in {-\N,...,\N}{
  \draw[black!66,line width=0.1pt]
    ($ -\N*(v3) + \k*(v1) $) --
    ($  \N*(v3) + \k*(v1) $);
}

\end{scope}

\fill[black!100]
  (0,0) --
  (1,0) --
  (0.5,{sqrt(3)/2}) --
  cycle;

\draw[->, thick, blue!20!black, >=Stealth] (0,0) -- ($(v1) - (v3)$) node[pos = 0.95, right] {$\alpha_1^\vee$};
\draw[->, thick, blue!20!black, >=Stealth] (0,0) -- ($(v2) + (v3)$) node[pos = 1, above] {$\alpha_2^\vee$};

\draw[dashed, gray!40, line width=0.5pt,
      dash pattern=on 5pt off 10pt]
  ( \H,0)
  -- ({\H/2},{\H*sqrt(3)/2})
  -- ({-\H/2},{\H*sqrt(3)/2})
  -- (-\H,0)
  -- ({-\H/2},{-\H*sqrt(3)/2})
  -- ({\H/2},{-\H*sqrt(3)/2})
  -- cycle;

\fill[black] (0,0) circle (2pt);

\end{tikzpicture}
\caption{An apartment of the Bruhat--Tits building of $\mathrm{PGL}_3$.}
\label{apart}
\end{figure}

\setcounter{namedtheorem}{1}
\begin{namedtheorem}[B]\label{ghnb}
Let $b$ be basic. If $x$ lies in the shrunken Weyl chambers then\[X_x(b)\neq\emptyset\text{ if and only if }\kappa_G(x)=\kappa_G(b)\text{ and }\sigmasupp(\etasigma(x))=\bbS.\]
\end{namedtheorem}
\noindent
Here, $\bbS$ is the set of simple reflections of $W_0$ and $\sigmasupp$ is the map sending $w\in W_0$ to the minimal $\sigma$-stable subset of $\bbS$ containing all simple reflections from any reduced expression of $w$ (\Cref{suppdefn}). Shrunken Weyl chambers are, intuitively, the complement of the red strips (called \textit{critical strips}) in \Cref{apart} and the critical strips are defined to be the strips passing through the base alcove (the black triangle).

\subsubsection{A side remark on explicit criteria}Here, we clarify the term `explicit criterion' mentioned in \cite[1.2]{GHN} and will be used in this paper at times. As this is not related to the rest of the paper, one may skip this remark.

In \textit{loc.cit.}, we have the following theorem (name it \Cref{ghna}).

\setcounter{namedtheorem}{0}
\begin{namedtheorem}[A]\label{ghna}
Let $b$ be basic. Then, $X_x(b)\neq\emptyset$ if and only if, for all pairs $(J,w)$ such that $x$ is a $(J,w,\delta)$-alcove, the following holds:\[\kappa_{M_J}(w^{-1}x\delta(w))\in \kappa_{M_J}\left([b]\cap M_J(\breve{\bbQ}_p)\right).\]\end{namedtheorem}
\noindent The Levi subgroup of $G$ corresponding to $J\subset \bbS$ is denoted by $M_J$ and $\kappa_{M_J}$ is its Kottwitz map. We refer to \Cref{jwsalcove} for the term \textit{$(J,w,\delta)$-alcove}.

Practically, \Cref{ghnb} is used often in applications\footnote{See, for example, \cite[Theorem 1.1]{He21} and \cite[Remark 3.18]{MilVie}. Using our new explicit criterion, we will give more applications.} while \Cref{ghna} is not. However, they solve the same problem and, actually, \Cref{ghnb} is more restrictive. Taking that into account, we can see that \Cref{ghnb} is more applicable and explicit already. Let us now see an example explaining this more clearly.

For the sake of simplicity, let $G$ be split or residually split for a moment. Let $x$ be a translation element $t^\mu$. Then, \Cref{ghnb} directly implies that, if $t^\mu$ lies in the shrunken Weyl chambers then $X_{t^\mu}(b)=\emptyset$ always.\footnote{This is because $\etasigma(x)$ is the identity element $\id_{W_0}$. See \Cref{etasigma} for $\etasigma(x)$.} Moreover, our new (explicit) criterion will show that $X_{t^\mu}(b)\neq\emptyset$ if and only if $[t^\mu]=[b]\in B(G)$. This recovers \cite[Corollary 9.2.1]{GHKR} for $b$ basic.

On the other hand, it is not easy to find all possible pairs $(J,w)$ such that $x$ is a $(J,w,\delta)$-alcove (especially \Cref{jwsalcove} (2) is not easily manageable) even when $x=t^\mu$, which is a necessary step to apply \Cref{ghna}. In addition, it is not an easy take to compute the values in $\kappa_{M_J}([b]\cap M_J(\breve{\bbQ}_p))$ afterwards.

\subsection{Main Conjecture}
Our goal is to remove the shrunken Weyl chambers condition. We will suggest a general conjecture on an explicit nonemptiness criterion in the basic case and prove it for all but finitely many $x$'s and specify some classes of elements satisfying the conjecture. Let $b$ be basic.

There is a new assumption in our conjecture that the following example can justify. Let $x=\id_{\extw}$, the identity element in $\extw$, and $b=1$. Then, directly from the definition, $X_x(b)\neq\emptyset$. However, $\sigmasupp(\etasigma(x))=\emptyset$ as $\etasigma(x)=\id_{W_0}$. Hence, there are some exceptional $x$'s not having the $\sigmasupp$-obstruction.

In the following, we denote by $\widetilde{\sigmasupp}$ the $\sigma$-support function on $\extw$.\footnote{We found that it may cause confusion to use the same notations for two $\sigma$-support functions each defined on $W_0$ and $\extw$ and decided to use $\widetilde{\sigmasupp}$ instead of $\sigmasupp:\extw\to 2^{\widetilde{\bbS}}$.}

\begin{lem}
\label{firstlemma}
Assume that $\kappa_G(x)=\kappa_G(b)$. If $\widetilde{\sigmasupp}(x)\neq\widetilde{\bbS}$ then $X_x(b)\neq\emptyset$.
\end{lem}
\begin{proof}
Let $\nu_x$ be the image of $x$ under the Newton map. The condition $\widetilde{\sigmasupp}(x)\neq\widetilde{\mathbb{S}}$ implies that $\nu_x$ is central because the group $W_{\widetilde{\sigmasupp}(x)}$ generated by the elements of $\widetilde{\sigmasupp}(x)$ is finite in such a case. However, if $\nu_x$ is central, for a representative $\dot{x}$, we have $\kappa(\dot{x})=\kappa(b)$ and $\bar{\nu}_{\dot{x}}=0=\bar{\nu}_b$ so that $[\dot{x}]=[b]\in B(G)$. Hence, $X_x(b)\neq\emptyset$.
\end{proof}
\Cref{firstlemma} shows that we only need to consider the case $\widetilde{\sigmasupp}(x)=\widetilde{\bbS}$. Our main conjecture is the following: (we follow the notations from \Cref{chtwo})

\begin{conj}\label{mainconj}
Let $G$ be a simple and quasi-split reductive group of adjoint type. Let $x$ be an element of Iwahori-Weyl group $\extw$ and $b\in\brG$ be basic. Assume that $\kappa_G(b)=\kappa_G(x)$ and $\widetilde{\sigmasupp}(x)=\widetilde{\mathbb{S}}$. Then,
\[X_x(b)\neq\emptyset\text{ if and only if }\sigmasupp(\sigma^{-1}(r)\eta_\sigma(x)r^{-1})=\mathbb{S}\text{ for all }r\in W_x,\]where $W_x$ is the subset of $W_0$ defined in \Cref{wxdefn}.
\end{conj}

\Cref{mainconj} claims that, except for the \Cref{firstlemma} cases, the critical obstruction for the nonemptiness is whether or not the $\sigma$-supports of certain $\sigma$-conjugates of $\etasigma(x)$ are all full. For convenience, we label by $\Leftarrow$ (resp. $\Rightarrow$) the `if' (resp. `only if') direction. Our main theorem is the following.
\begin{thm}\label{mainthm}Let $G$ be a reductive group as in \Cref{mainconj}. Then, the conjecture holds for all but finitely many $x$. More precisely, $\Leftarrow$ holds for all $x$ and $\Rightarrow$ holds for all but finitely many cases. In addition, $\Rightarrow$ holds for
\begin{enumerate}
    \item $x$ lying in exactly one critical strip (see \Cref{shrunken}),
    \item a translation element or $vt^\mu$-form element (for $\mu$ dominant), and
    \item when $G$ is of type $A_n$, and $x$ satisfies the condition of \Cref{Anresult}.
\end{enumerate}
\end{thm}
Note that \Cref{ghnb} takes care of infinitely many $x$'s, but it does not apply to infinitely many cases as well. We prove $\Leftarrow$ in \Cref{righttoleftgen} in full generality and prove $\Rightarrow$ in \Cref{lefttorightgen} for $x$ such that $\ell(x)\gg0$. An effective bound for the length can be computed and we do so for type $A_n$ in \Cref{Anresult}.

Let us discuss \Cref{mainthm} (3) in more detail. For simplicity, let $G=\PGL_n$. There are two assumptions in \Cref{Anresult}. The first assumption is that $\kappa_G(x)=0\in\mathbb{Z}/n\simeq X^*(Z(\hat{G})^\Gamma)$. It means that $x$ can be written as $t^\lambda w$ where $\lambda=(\lambda_1,\cdots,\lambda_n)\in \mathbb{Z}^n_{\text{sum}=0}$ and $w\in S_n$ (symmetric group). And then the second assumption is that `$\max_i{\lambda_i}>1$ and $\min_i{\lambda_i}<-1$'.

On the other hand, any element $x=t^\lambda w$ satisfying $\widetilde{\sigmasupp}(x)\neq\widetilde{\bbS}$ (cf. \Cref{firstlemma}) has the property `$\max_i\lambda_i=1$ and $\min_i\lambda_i=-1$'. Hence, there is only a small gap left in the conjecture under the assumption that $\kappa_G(x)=0$. This way, we can view \Cref{mainthm} (3) as an evidence for \Cref{mainconj} in the sense that elements in \Cref{firstlemma} are the only exceptions.

\begin{rem}
It is worth pointing out that \Cref{ghnb} and \Cref{mainthm} for $x$ lying in the shrunken Weyl chambers are \textit{almost} the same, but there is a subtle difference. For example, when $x=w_0$ (the longest element of $W_0$) as an element of $\extw$, we apply \Cref{ghnb} to get $X_x(1)\neq\emptyset$. However, $x$ does not fit into \Cref{mainthm} as $\widetilde{\sigmasupp}(x)\neq\widetilde{\bbS}$, so we should apply \Cref{firstlemma} to get $X_x(1)\neq\emptyset$ here.
\end{rem}

\begin{rem}
One might wonder what the condition `$\widetilde{\sigmasupp}(x)=\widetilde{\bbS}$' means. Surprisingly, a simple concrete picture exists for this condition in the sense that ``$\widetilde{\sigmasupp}(x)\neq\widetilde{\bbS}$ if and only if the action of $x\circ\sigma$ fixes a point in the closure of the base alcove''. Moreover, under the obvious obstruction (that is, $\kappa_G(x)=\kappa_G(b)$) when $b$ is basic, it is equivalent to that $I\dot{x}I\subset [b]$. (\cite[Proposition 5.6 and Lemma 5.8]{GHN19})
\end{rem}

Surprisingly, our result says that the abstract criterion (\Cref{ghna}) can be made much stronger as follows. This was observed by Sian Nie.

\begin{thm}[Stronger \Cref{ghna}]\label{thma}
    Let $b$ be basic and suppose that $\widetilde{\sigmasupp}(x)=\widetilde{\bbS}$ with $\ell(x)\gg0$. Then, $X_x(b)\neq\emptyset$ if and only if $x$ is not a $(J,w)_\sigma$-alcove for any proper $J\subsetneq\bbS$.
\end{thm}
We note that the condition `$\ell(x)\gg0$' can be removed once \Cref{mainconj} is fully proved. In short, for $x$ not intersecting with the base alcove, $X_x(b)\neq\emptyset$ if and only if $x$ is not ``coming from a proper Levi subgroup''.

\subsection{New Ideas and Sketch of Proof}The novelty of our work lies in the introduction of the set $W_x$ which arises naturally in the following sense.

The idea is that a critical strip behaves as if it belongs to the shrunken Weyl chambers \textit{adjacent} to the strip. Hence, if $x$ lies in a critical strip, we can ``embed'' $x$ into each shrunken Weyl chamber adjacent to the strip. Noting that, when $x$ lies in a shrunken Weyl chamber, the rule for $X_x(b)\neq\emptyset$ is $\sigmasupp(\etasigma(x))=\bbS$, we can presuppose that, in general, the rule for $X_x(b)\neq\emptyset$ would be $\sigmasupp(\etasigma(x'))=\bbS$ for each embedding $x'$ of $x$ into each shrunken Weyl chamber adjacent to the strips. 

The set $W_x$ is defined to be the set of elements in $W_0$ that, intuitively, embed $x$ into those shrunken Weyl chambers. Practically, for $x$ in the shrunken Weyl chambers, we get $W_x=\{\id_{W_0}\}$ and, for $x$ in one critical strip (cf. \Cref{mainthm} (1)), we have $W_x=\{\id_{W_0},s\}$ where $s$ is the simple reflection related to the critical strip containing $x$.

In order to define $W_x$, we make a new observation on a certain structure of the set of critical strips containing $x$ (\Cref{mainprop1}). More precisely, let $\Phi$ be the set of (relative) roots of $G$ and, for simplicity, $x$ lie in the dominant Weyl chamber. Next, denote by $\Phi_x$ the set of positive roots whose critical strip contains $x$. Then, $\Phi_x$ is ``anti-closed'' in the sense that if $\alpha+\beta\in\Phi_x$ for two positive roots $\alpha$ and $\beta$ then $\alpha\in \Phi_x$ or $\beta\in\Phi_x$. Alternatively, it is equivalent to that the complement of $\Phi_x$ in the set of positive roots is closed.

Using this, we show that the set $W_x$ is well-defined (\Cref{mainprop2}) and that $W_x$ contains exactly the elements needed for the `$\sigma$-support test' (cf. \Cref{suppinJgen}). After that, we follow the strategy of \cite{GHN} as explained in \Cref{knownresult} with more careful study on the $W_x$-action on dominant cocharacters. We also use the positivity of the coroot-coefficients of dominant cocharacters (see \Cref{oldpos}) and work with the exact (positive) coefficients in the proof of \Cref{mainthm} (1) and (3).

\begin{rem}
There has been an interesting coincidence happening while we have been working on this problem. Felix Schremmer (\cite{Sch}) recently defined a set called `length-positive' in his work on the generic Newton point. Our work and Schremmer's work do not overlap, that is, the results are rather complementary. However, the length-positive set of Schremmer $LP(x)$ is closely related to the set $W_x$ of ours. (See \Cref{LPandWx}.)
\end{rem}

\subsection{Applications and future works}\label{applications}

\subsubsection{The set $B(G)_x$ and cordial elements}
The question of whether $X_x(b)$ is nonempty is equivalent to the question of whether $I\dot{x}I\cap [b]$ is nonempty. In this perspective, we can denote the set of $[b]\in B(G)$ such that $I\dot{x}I\cap[b]\neq\emptyset$ by $B(G)_x$ and ask to describe it. This approach first appeared in \cite{Bea}.

\begin{rem}
For clarification, we remark that the main difference between the approach using $B(G)_x$ and ours is which variable is fixed. In our approach, we fix $b$ and study the nonemptiness, but the study of $B(G)_x$ fixes $x$. As we will see in a moment, these two approaches are complementary.
\end{rem}

When $x$ is cordial (see \Cref{cordial}), the set $B(G)_x$ has the property called saturated (see \Cref{saturated}), which makes the complete description of $B(G)_x$ easier. Combining this with \Cref{mainthm} (3), we obtain a full description of $B(G)_x$ in some special cases as follows.

\begin{thm}\label{bgx}
Let $x=vt^\mu$ for a dominant non-central\footnote{If $\mu$ is central, it is obvious that $B(G)_x=\{[t^\mu]\}=B(G,\mu)$ via \Cref{firstlemma}.} $\mu$ and $v\in W_0$, and let $W_0(\mu)$ be the stabilizer subgroup of $W_0$ fixing $\mu$. Then,\[W_x=\{r\in W_0{(\mu)}: \ell(vr^{-1})=\ell(v)+\ell(r)\}.\]Now, if $\sigmasupp(\sigma^{-1}(r)\etasigma(x)r^{-1})=\bbS$ for all $r\in W_x$ then $B(G)_x=B(G,\mu)$.
\end{thm}

\subsubsection{Future works on new dimension formulas}

The following dimension formula (\cite{He14}) is known for $x$ lying in the shrunken Weyl chambers:\[\dim X_x(b)=\frac{1}{2}\left(\ell(x)+\ell(\etasigma(x))-\defe_G(b)\right).\]
On the contrary, outside of the shrunken Weyl chambers, even a conjectural formula for $\dim X_x(b)$ is still mysterious. However, if more nonemptiness results are found, we can approach this problem with the following recursive formula (\cite[Proposition 4.2]{He14}): For $s\in \bbS$ satisfying $\ell(sx\sigma(s))=\ell(x)-2$,
\[\dim X_x(b)=\max\{\dim X_{sx}(b),\dim X_{sx\sigma(s)}(b)\}+1.\]
Typically, this is a bottom-up formula for the length reason (that is, $\ell(x)=\ell(sx)+1=\ell(sx\sigma(s))+2$). However, for example, if $X_{sx\sigma(s)}(b)$ is empty then we can also compute $\dim X_{sx}(b)$ from $\dim X_x(b)$, which is top-down.

In future work, using \Cref{mainthm} (2), we will show the following dimension formula which is new even in the rank two case:
\begin{thm}
Let $G$ be residually split with $\rk\! G_\textrm{sc}=2$ and $b$ be basic. For $x\in\extw$ lying in only one critical strip (associated to $v\alpha$), if $X_x(b)\neq\emptyset$ then \[\dim X_x(b)=\frac{1}{2}\left(\ell(x)+\min\{\ell(\eta_\sigma(x)),\ell(\sigma^{-1}(s_\alpha)\eta_\sigma(x)s_\alpha)\}-\defe_G(b)\right)-\epsilon,\]where $\epsilon=1$ if $\etasigma(x)=w_0$ and $\epsilon=0$ otherwise.
\end{thm}

\subsection{Organization}We describe the organization here.

\S2 is a preliminary section. We recall the setup of \cite{GHN} and note down some computations related to critical strips. Following \textit{loc.cit.}, we summarize the proof of \Cref{ghnb} and collect the lemmas subject to be generalized.

In \S3, we introduce the set $W_x$. Then, we discuss some properties of closed subsets of a root system (\cite{Djokovic}). Then we prove \Cref{mainthm} and \Cref{mainthm} (2).

In \S4, we prove \Cref{mainthm} (1) by handling some possibly exceptional cases. We study in detail using the exact positive coefficients from \cite{OnishchikVinberg}.

In \S5, we wrap up some tedious computations postponed in \S3 and \S4 proving \Cref{mainthm} (3). In the application part, we recall some facts from \cite{MilVie} and \cite{He21} and then prove \Cref{bgx}.

\subsection*{Acknowledgements}
This is a part of the author's thesis. The author thanks to his advisor Sug Woo Shin for his introduction to the Deligne-Lusztig world. The author also feels gratitude to Ulrich Görtz and Xuhua He for their kind replies and encouragement, and to Sian Nie and Felix Schremmer for their careful readings and comments on the first draft. Lastly, Michael Rapoport's comments on the history of affine Deligne-Lusztig varieties which the author was not aware of, are appreciated.

This project started when the pandemic covered the world and many ideas came to the author during the stay in Korea Institute for Advanced Study (KIAS). We would like to thank the institution for the support and hospitality.

\section{Basics on single affine Deligne-Lusztig varieties}\label{chtwo}

Let $F$ be a nonarchimedean local field with a uniformizer $t$ and $G$ be a connected reductive group over $F$. Denote by $\brF\colonequals\widehat{F^{\nr}}$ the completion of the maximal unramified extension of $F$ and by $\brG$ the $\brF$-points of $G$. The Frobenius map on the residue field of $\brF$ lifts to that of $\brF$ which we denote by $\sigma$ again. The induced map on $\brG$ will be denoted the same. Finally, we denote the set of $\sigma$-conjugacy classes of $\brG$ by $B(G)$.

\subsection{Iwahori-Weyl group and $B(G)$}\label{IwahoriWeyl}
Let $S$ be a maximal $\brF$-split torus of $G$ defined over $F$ and $T$ be the centralizer of $S$. Note that $T$ is a maximal torus of $G$ as $G$ becomes quasi-split over $\brF$ by Steinberg's Theorem.
\subsubsection{Iwahori-Weyl group $\extw$} The Iwahori-Weyl group associated to $S$ is defined as\begin{equation}\label{extw}
    \extw\colonequals N_S(G)(\Breve{F})/T(\Breve{F})_1
\end{equation}
where $T(\Breve{F})_1$ is the unique parahoric subgroup of $T(\Breve{F})$. We can fit $\extw$ into the following short exact sequence of groups (\cite{HainesRapoport} Definition 7)
\[0\to X_*(T)_{\Gamma_0}\to \extw \to W_0\to 1\]where ${\Gamma_0}$ is the absolute Galois group of $\brF$ and $W_0$ is the relative Weyl group\begin{equation}\label{relweyl}
    W_0\colonequals N_S(G)(\brF)/T(\brF)
\end{equation}

Now, we fix a $\sigma$-invariant base alcove $\mathbf{a}$ in the apartment of $S$ and let $I$ be the Iwahori subgroup of $G$ corresponding to $\mathbf{a}$. By fixing a special vertex in the closure of $\mathbf{a}$, we get a section $W_0\to \extw$ (not necessarily $\sigma$-equivariant) which allows us to identify $\extw$ with $X_*(T)_{\Gamma_0}\rtimes W_0$.

The Newton map $\nu:\extw\to X_*(T)_{\Gamma_0,\bbQ}^\sigma$ is defined as follows. The action of $\sigma$ on $\extw$ is of finite order and so, given $x\in\extw$, there exists a positive integer $N$ such that $\mu_N := \prod_{i=0}^{N-1}\sigma^i(x)$ belongs to $X_*(T)_{\Gamma_0}^\sigma$. The Newton map sends $x$ to $\nu_x = \mu_N/ N$ for any such $N$ and it does not depend on the choice of $N$.

Let $G_\textrm{sc}$ be the simply connected cover of the derived subgroup of $G$ and $T_\text{sc}$ the inverse image of $T$ via $G_\text{sc}\to G_\text{der}\to G$. The Iwahori-Weyl group of $G_\text{sc}$ is the affine Weyl group $W_a$ and gives rise to the following short exact sequence (\cite{HainesRapoport} Lemma 14):
\begin{equation}\label{bruhatorder}
    1\to W_a\to \extw\overset{\Tilde{\kappa}_G}{\longrightarrow} X^*(Z(\hat{G})^{\Gamma_0})\to1.
\end{equation}
Denoting by $\Omega\subset \extw$ the stabilizer of the base alcove, we have an isomorphism $X^*(Z(\hat{G})^{\Gamma_0})\simeq \Omega$ which gives a section of $\Tilde{\kappa}_G$. This presents $\extw$ as $W_a\rtimes X^*(Z(\hat{G})^{\Gamma_0})$. Now, the Bruhat order on $W_a$ extends onto $\extw$ by making two elements be comparable when their projections to $X^*(Z(\hat{G})^{\Gamma_0})$ agree.

Note that $W_a$ is the affine Weyl group generated by orthogonal reflections with respect to the hyperplanes in $X_*(T_\text{sc})_{\Gamma_0}\otimes\mathbb{R}$. Hence, by \cite[Ch.VI, \S2.5. Proposition 8]{Bourbaki}, there exists a reduced root system $\Sigma$ whose affine Weyl group is canonically isomorphic to $W_a$. We denote by $Q^\vee$ the coroot lattice of $\Sigma$ and $P^\vee$ its coweight lattice. Lastly, let $\mathbb{S}$ be the set of simple reflections of the finite Weyl group of $\Sigma$ (that is, $W_0$) through the section $W_0 \to \extw$ defined above and $\widetilde{\mathbb{S}}$ the set of affine simple reflections containing $\mathbb{S}$.

The map $\sigma$ on $\brG$ induces an action on $\bbS$ which we will denote by $\sigma$ again. We call $J\subset \bbS$ a $\sigma$-stable subset if $\sigma(J)=J$. For any $\sigma$-stable subset $J$, we denote $X_*(T)_{\Gamma_0}\rtimes W_J$ by $\extw_J$ where $W_J$ is the subgroup of $W_0$ generated by the simple reflections of $J$.

A comment on the notation: we will use the notation $v\cdot\mu$ when considering the $W_0$-action on $X_*(T)_{\Gamma_0}$. So, for example, $vt^\mu\in\extw$ can also be written as $t^{v\cdot\mu}v$.

\subsubsection{$B(G)$ with Newton map and Kottwitz map}
Recall the Newton map and the Kottwitz map from \cite[4.5]{Kottwitz} that give an injective homomorphism $(\bar{\nu},\kappa_G):B(G)\to X_*(T)_\mathbb{Q}^{\Gamma,+} \times X^*(Z(\hat{G})^{\Gamma})$ where $\Gamma$ is the absolute Galois group of $F$.

\begin{defn}[$B(G,\mu)$]\label{bgmu}
Let $\mu\in X_*(T)^+$ be a dominant cocharacter of $G$. We define $B(G,\mu)$ as the subset of $B(G)$ consisting of $[b]\in B(G)$ such that\[\bar{\nu}_b\le \mu^\diamond\text{ and }\kappa_G([b])=\kappa_G([t^\mu])\]where $\mu^\diamond$ is the $\Gamma$-average\footnote{To be precise, this is true only when $G$ is quasi-split. In general, it is the average of the \textit{dominant representatives} of Galois orbits. For a more detailed explanation, see \cite[2.4]{HeNie}.} of $\mu$ and $t^\mu$ is the image of $t$ under $\mu:\mathbb{G}_m\to T$. Mazur's inequality mentioned in \Cref{mazurineq} refers to $\bar{\nu}_b\le\mu^\diamond$.
\end{defn}

We note that the map $\Tilde{\kappa}_G$ in \Cref{bruhatorder} followed by the projection $X^*(Z(\hat{G})^{\Gamma_0})\to X^*(Z(\hat{G})^\Gamma)$ gives $\kappa_G:\extw\to X^*(Z(\hat{G})^\Gamma)$ and it is compatible with $\kappa_G$ on $B(G)$ via the lifting from $\extw$ to $N_S(G)(\brF)\subset\brG$.

\subsection{Single affine Deligne-Lusztig variety}\label{sadlv}
The Iwahori-Bruhat decomposition says
\[\brG=\bigsqcup_{x\in \extw} I\dot{x}I\]where $\dot{x}\in N_S(G)(\brF)$ is a representative of $x\in \extw$.
\begin{defn}[affine Deligne-Lusztig ``variety'']\label{adlvdef}
For $x\in \extw$ and $b\in B(G)$, the single affine Deligne-Lusztig variety associated to $x$ and $b$ is\[X_x(b)\colonequals\{gI\in\brG/I: g^{-1}b\sigma(g)\in I \dot{x} I \}.\]
\end{defn}
A priori, it is an affine Deligne-Lusztig \textit{set} and we do not have a natural scheme structure on it. In fact, in the case of an equal characteristic local field $F$, it is not difficult to identify $X_x(b)$ as the $\overline{\mathbb{F}}_q$-points of a (locally of finite type) locally closed subscheme in the affine flag variety over $\overline{\mathbb{F}}_q$. It is the mixed characteristic case where we need distinguished works of \cite{Zhu} and \cite{BS} to give a perfect scheme structure on $X_x(b)$.

In order to study the nonemptiness pattern, it is more convenient to do some reductions. By \cite[Corollary 4.4 and Section 4.3]{HeZhou}, we can reduce the nonemptiness problem to the case when $G$ is the quasi-split inner form of an adjoint group. Finally, if $G=G_1\times G_2$ then $B(G)=B(G_1)\times B(G_2)$ and $\extw_G=\extw_{G_1}\times\extw_{G_2}$. Hence, $X_x(b)= X_{x_1}(b_1)\times X_{x_2}(b_2)$ where $x_i$'s are the projections of $x$ onto $\extw_{G_i}$ and $b_i$'s are that of $b$ onto $B(G_i)$. Now, we may assume that $G$ is a simple quasi-split reductive group of adjoint type.

\begin{rem}
In \textit{loc.cit.}, it is assumed that $G$ is tamely ramified over $F$ and $p\nmid\!\pi_1(G_\text{ad})$ in the equal characteristic case. They refer to \cite{GHN} Proposition 2.2.1 which assumes, however, only that $p$ does not divide $\pi_1(G_\textrm{ad})$. Hence, we do not need to assume the tameness condition.

We can remove the $p\nmid\! \pi_1(G_\text{ad})$ condition as well because we do not need the isomorphism in Proposition 2.2.1 of \textit{loc.cit.} to study the nonemptiness. Note that they use \cite{PappasRapoport} 6.a.1 to show that the corresponding connected components of affine flag varieties of $G$ and $G_\text{ad}$ are isomorphic through the affine flag variety of $G_\text{sc}$ when $p\nmid\! \pi_1(G_\text{ad})$. However, a similar argument to Remark 6.4 of \textit{loc.cit.} actually shows that they still admit universal homeomorphisms from the affine flag variety of $G_\text{sc}$ without the assumption. This is enough for us to reduce to the case when $G$ is of adjoint type.
\end{rem}

\subsection{Terminologies on positions of alcoves}

From now on, we assume that $G$ is a simple quasi-split reductive group of adjoint type. For simplicity, we abusively use $x$ to denote the alcove $x\mathbf{a}$. For example, $\id_{\extw}$ denotes the base alcove.

Let $\Phi$ be the set $\Phi(G,S)$ of relative roots and $\Phi^+$ (resp. $\Phi^-$) be the subset of positive (resp. negative) roots. Let $V\colonequals X_*(T)_{\Gamma_0}\otimes\mathbb{R}$ which is isomorphic to $X_*(T_\text{sc})_{\Gamma_0}\otimes\mathbb{R}$ as $G$ is semisimple. For $\alpha\in \Phi$, the hyperplane $H_\alpha$ in $V$ is defined by $\{\mathbf{v}\in V:\langle \alpha,\mathbf{v}\rangle =0\}$ and, more generally, for any integer $k$, we define $H_\alpha(k)\colonequals\{\mathbf{v}\in V:\langle \alpha,\mathbf{v}\rangle=k\}$. Note that $H_\alpha(k)=H_{-\alpha}(-k)$.

\begin{defn}[{\cite[Section 2.1]{GHKR}}, $k$-value of an alcove with respect to a root]
For $\alpha\in \Phi$ and $x\in \extw$, let us denote the integer $k$ such that $x$ is located in between the hyperplanes $H_\alpha(k)$ and $H_\alpha(k+1)$ by $k(\alpha,x)$.
\end{defn}

\begin{defn}[The critical strips and shrunken Weyl chambers]\label{shrunken}
For each positive root $\alpha$, we call the set of alcoves between $H_\alpha(0)$ and $H_\alpha(1)$ the critical strip associated to $\alpha$ and denote by $CS_\alpha$.\footnote{One can define the critical strip of a negative root $\alpha$ to be the set of alcoves between $H_\alpha(0)$ and $H_\alpha(-1)$ and also denote by $CS_\alpha$. This convention will be used.} The set of alcoves which do not lie in any critical strip is called the shrunken Weyl chambers.
\end{defn}
The following computes the $k$-values explicitly. Here, the expression $t^\mu w$ is an alcove in the dominant Weyl chamber always and $v$ is an element of $W_0$.
\begin{lem}\label{klem} For any root $\alpha\in \Phi$,
\[k(\alpha,t^\mu w)=\left\{\begin{array}{ll}
    \langle \alpha,\mu \rangle & \text{if }w^{-1}\alpha>0, \\
    \langle \alpha,\mu \rangle-1 & \text{otherwise.}
\end{array}\right.\]
\end{lem}
\begin{proof}
We only need to consider the case where the alcove is represented by a finite Weyl group element and, in this case, the value $\delta_{w^{-1}\alpha}$ decides whether $w$ is in the $\alpha$-direction or $-\alpha$-direction.\end{proof}
In general, $k(\alpha,vt^\mu w)=\langle \alpha, v\mu\rangle+ \delta_{w^{-1}v^{-1}\alpha}$ where $\delta_\beta=0$ if $\beta\in \Phi^+$ and $\delta_\beta=-1$ otherwise.

\begin{cor}\label{kprop}Let $\alpha$ be a root in (1) and, $\alpha$, $\beta$, and $\alpha+\beta$ be roots in (2).
\begin{enumerate}
    \item $k(\alpha,t^\mu w)+k(-\alpha,t^\mu w)=-1$.
    \item $k(\alpha+\beta,t^\mu w)=k(\alpha,t^\mu w)+k(\beta,t^\mu w)$ or $k(\alpha+\beta,t^\mu w) = k(\alpha,t^\mu w)+k(\beta,t^\mu w)+1$.
\end{enumerate}
\end{cor}
\begin{proof}
(1): $w^{-1}\alpha>0$ implies $w^{-1}(-\alpha)<0$ and vice versa. (2): Apply \Cref{klem} noting that the cases where both $w^{-1}\alpha$ and $w^{-1}\beta$ are positive (resp. negative) but $w^{-1}(\alpha+\beta)$ is negative (resp. positive) are not possible.
\end{proof}
From \Cref{klem}, $vt^\mu w\in CS_{v\alpha}$ for some $\alpha\in \Phi^+$ if and only if $\langle \alpha,\mu\rangle+\delta_{w^{-1}\alpha}-\delta_{v\alpha}=0$. Noting that $k(\alpha,t^\mu w)\ge0$ when $\alpha\in\Phi^+$, we have the following corollary easily:
\begin{cor}
Let $\alpha\in\Phi^+$ and suppose $vt^\mu w\in CS_{v\alpha}$. Then $v\alpha\in\Phi^+$ and one of the following holds:
\begin{enumerate}
    \item $\langle\alpha,\mu\rangle=0$ and $w^{-1}\alpha\in\Phi^+$,
    \item $\langle\alpha,\mu\rangle=1$ and $w^{-1}\alpha\in\Phi^-$.
\end{enumerate}
When $vt^\mu w\not\in CS_{v\alpha}$, if $\langle\alpha,\mu\rangle=0$ then $v\alpha\in\Phi^-$ and $w^{-1}\alpha\in\Phi^+$.
\end{cor}
Now, we recall the definition of $(J,w)_\sigma$-alcove\footnote{In \cite{GHN}, this is denoted by $(J,w,\delta)$-alcove. We choose to use the notation `$(J,w)_\sigma$-alcove' in this paper not just for the sake of simplicity. In fact, the induced action $\delta$ on $\bbS$ by $\sigma$ is a fixed map. It is not varying, unlike $J$ or $w$ in the notation.} from \cite[3.3]{GHN}.
\begin{defn}[$(J,w)_\sigma$-alcove]\label{jwsalcove}Let $J$ be a $\sigma$-stable subset of $\bbS$ and $w$ be an element of $W_0$. We say that $x\in\extw$ is a $(J,w)_\sigma$-alcove if
\begin{itemize}
    \item[(1)] $w^{-1}x\sigma(w)\in \extw_J$ and,
    \item[(2)] for any $\alpha\in w(\Phi^+\setminus\Phi_J^+)$, $k(\alpha,x)\ge k(\alpha,\id)$.
\end{itemize}
\end{defn}
Next, the ``essential'' finite part $\etasigma(x)$ of $x$ observed by Reuman is defined as follows.
\begin{defn}[$\eta_\sigma(x)$, {\cite[3.6]{GHN}}]\label{etasigma}For $x\in \extw$, let $v_x\in W_0$ be the unique element such that $v_x^{-1}x$ is in the dominant Weyl chamber and $v_x^{-1}x=t^{\mu_x}w_x$ for $\mu_x\in X_*(T)_{\Gamma_0}^+$ and $w_x\in W_0$. We define \[\eta_\sigma(x)=\sigma^{-1}(w_x)v_x.\]
\end{defn}

Finally, we recall the following from \cite[Ch.IV, \S1.8. Proposition 7]{Bourbaki}.

\begin{defn}[support and $\sigma$-support]\label{suppdefn}Given a Coxeter system $(W,S)$ and $w\in W$, the support of $w$ is defined to be the set of $s\in S$ appearing in some (equivalently each) reduced expression of $w$ and denoted by $\supp(w)$. When $(W,S)$ is equipped with an action by $\xi$, the minimal $\xi$-stable set containing $\supp(w)$ is called the $\xi$-support of $w$ and denoted by $\supp_\xi(w)$.
\end{defn}

\begin{rem}
There are two Coxeter systems $(W_0,\bbS)$ and $(W_a,\widetilde{\bbS})$ in this paper. We have $\xi=\sigma$ for $(W_0,\bbS)$ and $\xi=\omega\sigma$ for $(W_a,\widetilde{\bbS})$ where $\omega$ is an element of $\Omega$ defined in \Cref{bruhatorder}. For $x\in W_a\rtimes \Omega$ whose projection to $\Omega$ is $\omega_x$, we use the notation $\widetilde{\sigmasupp}(x)$ instead of $\supp_{\omega_x\sigma}(x)$ for simplicity.
\end{rem}

\subsection{On the known result \Cref{ghnb}}\label{knownresult}

Now we recall:

\begin{thm}[\Cref{ghnb}]\label{GHNthm}Let $b\in B(G)$ be basic and $x\in \extw$ lie in the shrunken Weyl chambers with $\kappa(b)=\kappa(x)$. Then,\begin{center}$X_x(b)\neq\emptyset$ if and only if the $\sigma$-support of $\eta_\sigma(x)$ is $\mathbb{S}$.\end{center}
\end{thm}

\begin{proof}[Sketch of proof]
The proof  is a combination of the following lemmas.
\begin{lem}[{\cite[Proposition 3.6.4 and Proposition 3.6.5]{GHN}}]\label{nonemptyS}
    Let $b$ be basic. Under the following two assumptions, we have $X_x(b)=\emptyset$.
    \begin{enumerate}
        \item $\nu_{\mu_x}\neq\nu_{b}$ and
        \item $\sigmasupp(\eta_\sigma(x))\neq\mathbb{S}$,
    \end{enumerate}
    If we assume that $x$ lies in shrunken Weyl chambers, (2) implies (1).
\end{lem}
This proves that \textit{if $X_x(b)\neq\emptyset$ then $\sigmasupp(\eta_\sigma(x))=\mathbb{S}$}.\footnote{However, note that the proof of \Cref{nonemptyS} uses \cite[Proposition 3.5.1]{GHN} (``$\sigma$-conjugacy classes never fuse'') which was stated without the assumption that $b$ is basic but in fact needs that assumption. See \cite{ghnerror}.} For the reverse direction, we need the following theorem:
\begin{prop}[{\cite[Theorem 4.4.7]{GHN}}]\label{NLOnonempty}
Let $b$ be basic. If $x$ satisfies no Levi obstruction (NLO), then $X_x(b)\neq\emptyset$. Precisely, ``no Levi obstruction'' means the following: for every pair $(J,w)$ with $\sigma$-stable $J$ and $w\in W_0$ such that $x\in\extw$ is a $(J,w)_\sigma$-alcove, there exists $b_J\in w\extw_J\sigma(w)^{-1}$ such that
\begin{enumerate}
    \item $\kappa(b)=\kappa(b_J)$
    \item $\nu_{b_J}=\nu_b$
    \item $\kappa_J(w^{-1}b_J\sigma(w))=\kappa_J(w^{-1}x\sigma(w))$
\end{enumerate}
\end{prop}
If $x$ lies in the shrunken Weyl chamber and $\sigmasupp(\eta_\sigma(x))=\mathbb{S}$, by \Cref{suppinJ} below, the only $J$ such that $x$ is a $(J,w)_\sigma$-alcove is $J=\mathbb{S}$. For $J=\mathbb{S}$, we can let $b_J$ be any element in the $\sigma$-straight conjugacy class of $\extw$ corresponding to $[b]\in B(G)$ (\cite[3.3]{He14}) so that $x$ satisfies NLO and $X_x(b)\neq\emptyset$ by \Cref{NLOnonempty}.
\end{proof}

\begin{lem}[{\cite[Proposition 4.1.1]{GHN}}]\label{suppinJ}Let $x\in\extw$ lie in the shrunken Weyl chambers. If $x$ is a $(J,w)_\sigma$-alcove for a $\sigma$-stable subset $J$ of $\mathbb{S}$, then $\sigmasupp(\eta_\sigma(x))\subset J$.
\end{lem}

Finally, we have the following lemmas on the properties of dominant cocharacters and the Weyl group action on roots. For a subset $J$ of $\mathbb{S}$, we denote the lattice generated by the coroots $\alpha_i^\vee$ such that $s_i\in J$ by $Q_J^\vee$.
\newpage
\begin{lem}\label{wmu-mu}
    Let $\mu$ be a cocharacter and $w\in W_0$. Then, $\mu-w\cdot\mu\in Q_{\supp(w)}^\vee$.
\end{lem}
\begin{proof}
For any simple reflection $s_i$, the formula $s_i\cdot\mu=\mu-\langle\alpha_i,\mu\rangle\alpha_i^\vee$ tells us that $\mu-s_i\cdot\mu\in Q_{\{s_i\}}^\vee$. Note that $w$ is a product of simple reflections and so we can think inductively on $\ell(w)$. Considering $\mu-ws_i\cdot\mu=\mu-s_i\cdot\mu+ s_i\cdot\mu-w\cdot(s_i\cdot\mu)\in Q_{\{s_i\}}^\vee + Q_{\supp(w)}^\vee$, it is straightforward.
\end{proof}
The following corollary is stated in \cite[3.5 (1)]{GHN} without a proof. For completeness, we give a proof for this.
\begin{cor}\label{GHN351}
    Let $J$ be a $\sigma$-stable subset of $\mathbb{S}$. If $x\in \extw_J$, then $\nu_{v_x\cdot\mu_x}-\nu_x\in Q_{J,\bbQ}^\vee$.
\end{cor}

\begin{proof}
The difference is a $\bbQ$-multiple of $\sum_{i}(\sigma^i(v_x\cdot\mu_x)-(v_xw_x\sigma)^i(v_x\cdot\mu_x))$. Noting that $\prod_{0\le j\le i-1}\sigma^j(v_xw_x)\in W_J$, apply \Cref{wmu-mu} to get the conclusion.
\end{proof}

Lastly, we prove:
\begin{lem}\label{wronglem}
Suppose that the Dynkin diagram of $G$ is $\sigma$-connected. Let $J$ be a proper $\sigma$-stable subet of $\mathbb{S}$. If $\mu\in Q^\vee_{J,\mathbb{Q}}$ is dominant, then $\mu=0$.
\end{lem}
\begin{proof}
$Q^\vee_{J,\mathbb{Q}}\subset Q^\vee_{\mathbb{Q}}=P^\vee_{\mathbb{Q}}$, hence $\mu$ is a $\mathbb{Q}$-linear combination of fundamental coweights. As it is dominant, all the coefficients are non-negative. If $\mu\neq0$, then at least one coefficient is positive so that it is a positive linear combination of coroots of a connected component by \Cref{oldpos}. Hence, $J$ contains some connected component completely. As the Dynkin diagram is $\sigma$-connected, we must have $J=\mathbb{S}$ which is contradiction.
\end{proof}

The following \Cref{oldpos} has been probably known to experts for a long time. However, to the best of our knowledge, a reference for this lemma is hard to locate. So, following Michael Rapoport's suggestion, we record it here. In this paper, this is particularly needed to remedy some argument in \cite{GHN} using \textit{loc.cit.} 3.5. (2) which is not correct. See \Cref{nonemptysetaS}.

\begin{lem}\label{oldpos}\footnote{The letter $C$ will be used in this lemma to follow the historical notations. The letter $C$ appearing in this paper is always referring to the critical strip, except here. We believe there would be no confusion occurring.}
    Let $\bar{C}$ be the closed Weyl chamber defined as $\{\mu\in V: \langle \mu,\alpha\rangle\ge0\text{ for all }\alpha\in\Delta\}$ and $C^\vee$ be the obtuse Weyl chamber defined as $\{\mu\in V:\mu=\sum_{i\in\bbS}c_i\alpha_i^\vee,~c_i\ge0\}$. Then $\bar{C}\setminus\! \{\vec{0}\}$ is contained in the interior of $C^\vee$. In other words, a fundamental coweight is a positive linear combination of simple coroots.
\end{lem}

\begin{proof}
    The fundamental coweights $\varpi_i^\vee$ form a basis of $V$ and they satisfy $\langle\varpi_i^\vee,\alpha_j\rangle=\delta_{ij}$. Hence, $\bar{C}\setminus\{\vec{0}\}=\{\sum_{i\in\mathbb{S}}d_i\varpi_i^\vee:\text{ not all }d_i\text{'s are }0\}$. Hence, it is enough to show that $\langle \varpi_i^\vee,\varpi_j\rangle\equalscolon \pi_{ij}$, the coefficient of $\alpha_j^\vee$ in the expression of $\varpi_i^\vee$, is positive.
        
    First of all, the coefficients from \cite[Ch.VI, \S4.5-\S4.13, (VI)]{Bourbaki} prove the dual version. For our version, recall the Cartan matrix $A$ whose entries are $A_{ij}=\langle\alpha_i^\vee,\alpha_j\rangle$. Then, $\delta_{ik}=\langle \varpi_i^\vee,\alpha_k\rangle=\langle \sum_j\pi_{ij}\alpha_j^\vee,\alpha_k\rangle= \sum_j\pi_{ij}A_{jk}$. Therefore, $\pi_{ij}$'s are the entries of the inverse of $A$, which we know to be positive by \cite[Reference Chapter. §2. Table 2]{OnishchikVinberg} (note that one should take their transpose). For a more general situation, see \cite[5.]{LusztigTits}.
\end{proof}

\section{Proof of main theorems}
\subsection{Closed subsets of a root system}

\begin{defn}
    A subset $\Psi$ of the set of roots $\Phi$ is called\footnote{The terms \textit{closed} and \textit{parabolic} are from \cite[Ch.VI, \S1.7. Définition 4.]{Bourbaki} and the term \textit{radical} is from \cite[1.]{Djokovic}.}
    \begin{enumerate}
        \item closed if $\alpha,\beta\in\Psi$ implies $\{\alpha+\beta\}\cap\Phi\subset\Psi$,
        \item radical if $\Psi\cap-\Psi=\emptyset$, and
        \item parabolic if $\Psi\cup-\Psi=\Phi$.
    \end{enumerate}
\end{defn}
Being closed or radical are invariant under any $W_0$-action. An example of a radical closed subset is any subset of all positive roots or $W_0$-conjugate of that. In fact, this is essentially the only example.
\begin{lem}\label{radicalclosedispositive}
    For a radical closed subset $\Psi$, there exists $w\in W_0$ such that $\Psi\subset w\Phi^+$.
\end{lem}
\begin{proof}
\cite[Proposition 3]{Sopkina} (cf. \cite[Ch.VI, \S1.7. Proposition 22]{Bourbaki})
\end{proof}
The following proposition is stronger than \Cref{radicalclosedispositive} and crucial to generalize our work to multiple critical strips case. The proof relies on the classification of parabolic subsets.
\begin{prop}\label{mainprop2}
    Let $\Psi_r\subset\Psi_p$ be two closed subsets of $\Phi$ such that $\Psi_r$ is radical and $\Psi_p$ is parabolic. Then, there exists $w\in W_0$ such that \[w\Psi_r\subset\Phi^+\subset w\Psi_p.\]
\end{prop}
\begin{proof}
We call a closed subset invertible if the complement (in the full $\Phi$) is also closed. A parabolic subset is an invertible subset and all invertible subsets are $W_0$-conjugate to $\Phi_J \cup(\Phi^+\setminus \Phi_{J'}^+)$ where $J\subset J'\subset\bbS$ and $J\bot (J'\setminus J)$ by \cite[Lemma 1 and Theorem 4]{Djokovic}. Only when $J=J'$, it is parabolic.

Hence, $\Psi_p=w_0 (\Phi^+\cup \Phi_J^-$) for some $w_0\in W_0$ and $J\subset\bbS$. Now, consider $w_0^{-1}\Psi_r\cap \Phi_J$ which is radical because it belongs to a radcial subset $w_0^{-1}\Psi_r$ and closed because it is the intersection of two closed subsets. Hence, there exists $w_1\in W_J$ such that $w_0^{-1}\Psi_r\cap\Phi_J\subset w_1 \Phi_J^+$ by \Cref{radicalclosedispositive}.

Note that $\Psi_r\subset\Psi_p=w_0(\Phi^+\cup \Phi_J^-)$ so that \[w_1^{-1}w_0^{-1}\Psi_r=w_1^{-1}(w_0^{-1}\Psi_r\cap \Phi_J) \cup w_1^{-1}(w_0^{-1}\Psi_r \cap (\Phi^+\setminus\Phi_J^+))\]But from above the first part belongs to $W_J^+$. The second set belongs to $w_1^{-1}(\Phi^+\setminus\Phi_J^+)$ which belongs to $\Phi^+$ as $w_1\in W_J$. Obviously, $w_1^{-1}w_0^{-1}\Psi_p=w_1^{-1}(\Phi^+\cup\Phi_J^-)\supset \Phi^+$.
\end{proof}

Recall the notation $x= v_x t^{\mu_x}w_x $ from \Cref{etasigma} and let $\Phi_x$ be the set of positive roots $\alpha$ such that $x\in CS_{v_x\alpha}$. The following is the most important observation towards our main theorem.

\begin{prop}\label{mainprop1}
$\Phi^+\setminus\Phi_x$ is a radical closed subset.
\end{prop}
\begin{proof}
It is enough to prove that $\Psi_x$ satisfies the following:
\[\text{for }\alpha,\beta\in\Phi^+\text{, if }\alpha+\beta\in \Phi_x\text{ then }\alpha\in\Phi_x\text{ or }\beta\in\Phi_x.\]
However, for a positive root $\gamma$, $\gamma\in\Psi_x$ is equivalent to $k(\gamma,t^{\mu_x}w_x)=k(v_x\gamma,\id)=0$. Applied to $\gamma=\alpha+\beta$, we get $k(\alpha+\beta,t^{\mu_x}w_x)=k(v_x\alpha+v_x\beta,\id)=0$. Using \Cref{kprop} (2) and that $t^{\mu_x}w_x$ lies in the dominant Weyl chamber, we get $k(\alpha,t^{\mu_x}w_x)=k(\beta,t^{\mu_x}w_x)=0$. Moreover, by \Cref{kprop} (2) again, $0=\delta_{v_x\alpha}+\delta_{v_x\beta}+1$ or $0=\delta_{v_x\alpha}+\delta_{v_x\beta}$. Hence, $\delta_{v_x\alpha}=0$ or $\delta_{v_x\beta}=0$, that is, $\alpha\in \Psi_x$ or $\beta\in\Psi_x$.
\end{proof}

\subsection{The set $W_x$}

We define $W_x$ mentioned in \Cref{mainconj}.

\begin{defn}\label{wxdefn}
Given $x\in\extw$, the subset $W_x\subset W_0$ is the set of $r\in W_0$ such that\[r(\Phi^+\setminus\Phi_x)\subset\Phi^+,\]or equivalently, by taking the negation and the complement, $r^{-1}\Phi^+\subset \Phi^+\cup -\Phi_x$.
\end{defn}
\begin{rem}
If $x$ lies in a shrunken Weyl chamber then $\Phi_x=\emptyset$, so $W_x=\{\id_{W_0}\}$. For $x$ lying in exactly one critical strip, we have $\Phi_x=\{\alpha_x\}$ for some unique $\alpha_x\in\Phi^+$. In fact, it should be a simple root by \Cref{mainone}. Hence,\[W_x=\{r\in W_0: r^{-1}\Phi^+\subset\Phi^+\cup \{-\alpha_x\}\}=\{id,s_x\}\]where $s_x$ is the simple reflection corresponding to $\alpha_x$.
\end{rem}

We note that $W_x$ is not necessarily a subgroup, but the following lemma suggests some structure on $W_x$.
\begin{lem}\label{wxleftclosed}
    The set $W_x$ is left-closed in the sense that if $w\in W_x$ and $\ell(sw)<\ell(w)$ for a simple reflection $s$, then $sw\in W_x$.
\end{lem}

\begin{proof}
We know that $w$ can be written as $sw'$ where $\ell(sw')=\ell(w')+1$. Then $w'^{-1}\Phi^+=w^{-1}s(\Phi^+)=w^{-1}(\Phi^+\cup\{-\alpha_s\}\setminus\{\alpha_s\})=w^{-1}\Phi^+\cup \{-w\alpha_s\}\setminus\{w\alpha_s\}$. However, as $w^{-1}s<w^{-1}$, $-w^{-1}\alpha_s\in\Phi^+$ so that $w^{-1}\Phi^+\cup\{-w\alpha_s\}\setminus\{w\alpha_s\}\subset w^{-1}\Phi^+\cup\Phi^+\subset \Phi^+\cup-\Phi_x$ as $w\in W_x$.
\end{proof}

As promised previously, we show that $LP(x)$ from \cite{Sch} is closely related to $W_x$.

\begin{lem}\label{LPandWx}
    We have \[W_x = LP(x)^{-1} v_x,\]where $LP(x)^{-1} = \{v \in W_0:v^{-1} \in LP(x)\}$.
\end{lem}

\begin{proof}
    For $v\in W_0$, we have $v \in LP(x)$ if and only if, for all $\alpha \in \Phi^+$, we have $k(v\alpha, x) \ge k(v\alpha, \id)$. If $v_x^{-1}v\alpha$ is a negative root, we have $k(v\alpha, x) = k(v_x v_x^{-1}v\alpha, v_xt^{\mu_x}w_x) = k(v_x^{-1}v\alpha, t^{\mu_x}w_x) < 0$. However, $k(v\alpha, \id)$ is always $0$ or $-1$, which means that $k(v\alpha, x) = k(v\alpha,\id) = -1$ and so $x$ is in $CS_{v\alpha} = CS_{-v\alpha} = CS_{v_x{(-v_x^{-1}v\alpha)}}$ and, by definition, $-v_x^{-1}v\alpha\in \Phi_x$. Therefore, we have that $v\in LP(x)$ implies $v_x^{-1}v\Phi^+ \subset \Phi^+ \cup -\Phi_x$ and $(v_x^{-1}v)^{-1} = v^{-1}v_x \in W_x$. The reverse direction works in a similar way.
\end{proof}

\subsection{\Cref{suppinJ} in general}
As before, $b$ is basic. We will express $x$ as $v_xt^{\mu_x}w_x$ where $v_x\in W_0$ is the unique element such that $v_x^{-1}x$ is in the dominant Weyl chamber.

\begin{lem}\label{suppinJgen}
    If $x$ is a $(J,w)_\sigma$-alcove for a $\sigma$-stable subset $J$ of $\bbS$, then $\sigma^{-1}(r)\etasigma(x)r^{-1}\in W_J$ for some $r\in W_x$.
\end{lem}
Similarly as before, we get the following corollary:
\begin{cor}\label{righttoleftgen}
    Let $x$ be an arbitrary element in $\extw$ and $b\in \brG$ such that $\kappa(x)=\kappa(b)$. If $\sigmasupp(\sigma^{-1}(r)\etasigma(x)r^{-1})=\bbS$ for all $r\in W_x$, then $X_x(b)\neq\emptyset$.
\end{cor}
\begin{proof}[Proof of \Cref{suppinJgen}]
Suppose that there exists $\alpha\in w(\Phi^+\setminus\Phi_J^+)$ such that $v_x^{-1}\alpha\in -(\Phi^+\setminus \Phi_x)$. Then, by definition of $\Phi_x$, $-\alpha$ satisfies $k(-\alpha,x)\neq k(-\alpha,\id)$, that is, $k(\alpha,x)\neq k(\alpha,\id)$. Hence, $k(v_x^{-1}\alpha,t^{\mu_x}w_x)\gneq k(\alpha,\id)\ge-1$ and so $v_x^{-1}\alpha$ should be a positive root as $t^{\mu_x}w_x$ is in the dominant Weyl chamber. This contradicts to the assumption. Hence, $v_x^{-1}\alpha\in \Phi^+\cup -\Phi_x$.

This implies that $v_x^{-1}w(\Phi^+\setminus\Phi_J^+)\subset \Phi^+\cup -\Phi_x$, or equivalently, $\Phi^+\setminus\Phi_x\subset v_x^{-1}w(\Phi^+\cup\Phi_J^-)$. By \Cref{mainprop1}, $\Psi_r\colonequals\Phi^+\setminus \Phi_x$ and $\Psi_p\colonequals v_x^{-1}w(\Phi^+\cup\Phi_J^-)$ fit into the assumptions in \Cref{mainprop2} so that there exists $r\in W_0$ such that\[r(\Phi^+\setminus\Phi_x)\subset \Phi^+\subset rv_x^{-1}w(\Phi^+\cup \Phi_J^-).\]
Taking the complement and the negation, we get $rv_x^{-1}w(\Phi^+\setminus \Phi_J^+)\subset \Phi^+$ for some $r\in W_x$ by \Cref{wxdefn}. Therefore, $rv_x^{-1}w\in W_J$ for some $r\in W_x$ and consequently $\sigma(w^{-1})\sigma(v_x)\sigma(r^{-1})\in W_J$ as $J$ is $\sigma$-stable. Moreover, $w^{-1}v_xw_x\sigma(w)\in W_J$ since $x$ is a $(J,w)_\sigma$-alcove. Multiplying them together, we get $rw_x\sigma(v_x)\sigma(r^{-1})\in W_J$ so that $\sigma^{-1}(r)\etasigma(x) r^{-1}\in W_J$.
\end{proof}

\subsection{\Cref{nonemptyS} under some restriction}
For $r\in W_x$, let us denote by $J_{r,x}$ the $\sigma$-stable subset $\sigmasupp(\sigma^{-1}(r)\etasigma(x)r^{-1})$.
\begin{lem}\label{jrxalcove}
    $x$ is a $(J_{r,x},v_xr^{-1})_\sigma$-alcove.
\end{lem}
\begin{proof}
    The first condition of \Cref{jwsalcove} can be checked easily because the finite part is $rv_x^{-1}v_xw_x\sigma(v_xr^{-1})\in W_{J_{x,r}}$. For \Cref{jwsalcove} (2), we need to compare $k(v_xr^{-1}\alpha,v_xt^{\mu_x}w_x)$ and $k(v_xr^{-1}\alpha,\id)$ for all $\alpha\in\Phi^+\setminus\Phi_{J_{x,r}}^+$. The first one is $k(r^{-1}\alpha,t^{\mu_x}w_x)$ which is $\ge0$ if $r^{-1}\alpha\in\Phi^+$ and $0\ge k(v_xr^{-1}\alpha,\id)$ always. Hence, we only need to consider the case $r^{-1}\alpha\in\Phi^-$. However, $r(\Phi^+\setminus \Phi_x)\subset \Phi^+$ implies (again by the negation complement) $r^{-1}\Phi^+\subset \Phi^+\cup -\Phi_x$. Therefore, we have $r^{-1}\alpha\in -\Phi_x$. By the definition of $\Phi_x$, we have the same $k$-values.
\end{proof}

\begin{prop}\label{lefttorightgen}
    Let $x\in\extw$ be an element with $\ell(x)\gg0$. If $J_{r,x}\neq\bbS$ for some $r\in W_x$, then $X_x(b)=\emptyset$.
\end{prop}

\begin{proof}
    By the previous lemma, $x$ is a $(J_{r,x},vr^{-1})_\sigma$-alcove so \cite[Proposition 3.6.4]{GHN} tells us that\[\nu_{r\cdot\mu_x}\in Q_{J_{r,x},\bbQ}^\vee.\]Note that, inductively computing, when $r=s_{i_m}\cdots s_{i_1}$,
    \begin{equation}\label{rmu-mu}
        r\cdot\mu_x=\mu_x-\sum_{j=1}^k\langle s_{i_1}\cdots s_{i_{j-1}}\alpha_{i_j},\mu_x\rangle\alpha_{i_j}^\vee
    \end{equation}so that\[\nu_{r\cdot\mu_x}=\nu_{\mu_x}-\nu_{\langle \alpha_1,\mu_x\rangle\alpha_1^\vee+\cdots+\langle s_1\cdots s_{k-1}\alpha_k,\mu_x\rangle\alpha_k^\vee}.\]Assume that the expression of $r$ is reduced, then $s_{i_1}\cdots s_{i_{j-1}}\alpha_{i_j}\in\Phi^+$. Moreover, $rs_{i_1}\cdots s_{i_{j-1}}\alpha_{i_j}\in\Phi^-$. Hence, $s_{i_1}\cdots s_{i_{j-1}}\alpha_{i_j}\in \Phi^+\cap r^{-1}\Phi^-\subset\Phi_x$ implying that $\langle s_{i_1}\cdots s_{i_{j-1}}\alpha_{i_j},\mu_x\rangle=0$ or $1$. Therefore, the $\sigma$-average of the sum of coroots on the right hand side is bounded. However, for $x$ such that $\ell(x)\gg0$, if $\mu_x=\sum_ic_i\varpi_i^\vee$, then $\sum_ic_i\gg0$ and, by \Cref{oldpos}, the coefficients in the linear combination of coroots are large enough so that $\nu_{\mu_x}-\nu_{\sum_{j=1}^k\langle s_{i_1}\cdots s_{i_{j-1}}\alpha_{i_j},\mu_x\rangle\alpha_{i_j}^\vee}$ has all positive coefficients. This is a contradiction.
\end{proof}
Combining \Cref{lefttorightgen} and \Cref{righttoleftgen}, we get
\begin{thm}\label{maintwo}Let $b\in B(G)$ be basic and suppose that $x\in \extw$ satisfies $\kappa(b)=\kappa(x)$ and $\ell(x)\gg0$. Then\begin{center}$X_x(b)\neq\emptyset$ if and only if $\sigmasupp(\sigma^{-1}(r)\eta_\sigma(x)r^{-1})=\mathbb{S}$ for all $r\in W_x$.
\end{center}
\end{thm}
\begin{proof}[Proof of \Cref{mainthm}]
Note that $\widetilde{\sigmasupp}(x)=\widetilde{\bbS}$ condition in \Cref{mainthm} (1) is satisfied for $x$ such that $\ell(x)\gg0$. As there are only finitely many $x$ with $\ell(x)$ bounded by some number, \Cref{mainconj} holds.\end{proof}
\begin{proof}[Proof of \Cref{mainthm} (2)]\label{123proof}
If $x=v_xt^{\mu_x}w_x$ is a translation element, we have $w_x=v_x^{-1}$. In this case, $\Phi_x$ consists of $\alpha\in\Phi^+$ such that $0\le\langle\alpha,\mu_x\rangle+\delta_{v_x\alpha}=\delta_{v_x\alpha}$ so that $\langle\alpha,\mu_x\rangle=0$ and $v_x\alpha\in \Phi^+$ for all $\alpha\in\Phi_x$. For $x=v_xt^{\mu_x}$, we have $w_x=\id$ so that $\alpha\in\Phi_x$ satisfies $\langle\alpha,\mu_x\rangle+\delta_{\alpha}=\delta_{v_x\alpha}$ so that $\langle\alpha,\mu_x\rangle=0$ and $v_x\alpha\in\Phi^+$ as well.

Hence, in both cases, we do not have $\alpha\in\Phi_x$ such that $\langle\alpha,\mu_x\rangle=1$. Therefore, in the proof of \Cref{lefttorightgen}, all the terms $\langle s_{i_1}\cdots s_{i_{j-1}}\alpha_{i_j},\mu_x\rangle\alpha_{i_j}^\vee$ are zeros. Therefore, $\nu_{r\cdot\mu_x}$ is dominant still so that $J_{r,x}\neq\bbS$ is a contradiction assuming $X_x(b)\neq\emptyset$.
\end{proof}

\begin{proof}[Proof of \Cref{thma}]By \Cref{suppinJgen} and \Cref{jrxalcove}, $x$ is a $(J,w)_\sigma$-alcove if and only if $J=J_{r,x}$. \Cref{maintwo} implies that, for $\ell(x)\gg0$, $X_x(b)\neq\emptyset$ if and only if $J_{r,x}=\bbS$ for all $r\in W_x$.\end{proof}


\section{One critical strip}
\subsection{Finding a simple root}
We will prove the following:

\begin{thm}\label{mainone}Let $b\in B(G)$ be basic and $x\in \extw$ lie in exactly one critical strip $CS_a$ for some $a\in\Phi$. Then
\begin{enumerate}
    \item $v_x^{-1}a$ is a simple root, and
    \item $X_x(b)\neq\emptyset$ if and only if $\kappa_G(b)=\kappa_G(x)$ and \begin{center}both $\sigmasupp(\etasigma(x))$ and $\sigmasupp(\sigma^{-1}(s_x)\etasigma(x)s_x)$ are $\mathbb{S}$,\end{center}where $s_x$ is the simple reflection corresponding to $v_x^{-1}a$.
\end{enumerate}
\end{thm}

This is slightly stronger than \Cref{mainthm} (1). More precisely, compared to \Cref{mainthm} (1), \Cref{mainone} has additional claims that $v_x^{-1}a$ is a simple root and has no restriction `$\widetilde{\sigmasupp}(x)=\widetilde{\bbS}$'. In fact, one can easily show that $v_x^{-1}a$ must be a simple root (\Cref{onestripsimple}) from \Cref{mainprop1}. The absence of the restriction $\widetilde{\sigmasupp}(x)=\widetilde{\bbS}$ is the main reason why \Cref{mainone} is stronger. It is explained in \Cref{nucentral}.

\begin{lem}\label{onestripsimple}
Suppose that $x$ belongs to exactly one critical strip associated to $a$. Then, $v_x^{-1}a$ is a simple root.
\end{lem}
\begin{proof}
By \Cref{mainprop1}, $\Phi_x$ is a singleton such that $\Phi^+\setminus\Phi_x$ is closed. As simple roots generate $\Phi^+$, we know $\Phi_x$ should contain a simple root.\end{proof}



We will denote by $\alpha_x$ the unique positive simple root such that $x\in CS_{v_x\alpha_x}$.

\subsection{\Cref{nonemptyS} for one strip}
Here, we prove that if $\sigmasupp(\eta_\sigma(x))\neq\mathbb{S}$ or $\sigmasupp(s_x\eta_\sigma(x)\sigma(s_x))\neq\mathbb{S}$, then $X_x(b)=\emptyset$.

\begin{prop}\label{nonemptyetaS}Under the assumption of \Cref{mainone}, if $\sigmasupp(\eta_\sigma(x))\neq\mathbb{S}$ then $X_x(b)=\emptyset$.
\end{prop}
\begin{proof}By \Cref{nonemptyS}, it is enough to show that $\nu_{\mu_x}\neq\bar{\nu}_b$. We follow the strategy of \cite[Proposition 3.6.5]{GHN} here.

Let $n$ be the order of $W_0\rtimes\langle\sigma\rangle$ and suppose that $\nu_{\mu_x}$ is central. Then $0=\langle\nu_{\mu_x},\beta\rangle=\frac{1}{n}\langle\mu_x+\sigma(\mu_x)+\cdots+\sigma^{n-1}(\mu_x),\beta\rangle$, but $\sigma^i(\mu_x)$ are all dominant. Hence, if $\beta$ is positive then $\langle\mu_x,\beta\rangle=0$. Consider the case $\beta$ being the maximal root, we get $\mu_x$ central. This implies that $w_x=\id$ and $x=v_xt^{\mu_x}$. Note that $\eta_\sigma(x)= v_x$.

As we assume that there is only one critical strip containing $x$, it means that the number of $\alpha\in \Phi^+$ such that $k(\alpha,v_xt^{\mu_x})= k(\alpha,\id)=0$ is $1$. But, $k(\alpha,v_xt^{\mu_x})=\delta_{v_x^{-1}\alpha}$. It means that $|\Phi^+\cap v_x \Phi^+|=|\Phi^+|-\ell(v_x)=1$ so that $\ell(v_x)=|\Phi^+|-1$. However, if $\ell(v_x)\le |\Phi_{\sigmasupp(v_x)}^+|<|\Phi^+|-1$.
\end{proof}

Now we consider the case $\sigmasupp(\sigma^{-1}(s_x)\eta_\sigma(x)s_x)\neq\mathbb{S}$.
\begin{lem}\label{xisJalcove}
Under the assumption of \Cref{mainone}, denote by $J_x$ the set $\sigmasupp(\sigma^{-1}(s_x)\eta_\sigma(x)s_x)$. Then $x$ is a $(J_x,v_xs_x)_\sigma$-alcove.\end{lem}
\begin{proof}
We have $\Phi_x=\{\alpha_x\}$, so $s_x\in W_x$ because $s_x\Phi^+=\Phi^+\cup\{-\alpha_x\}\setminus\{\alpha_x\}$. Now, the conclusion follows from \Cref{jrxalcove}.
\end{proof}
\begin{prop}\label{nonemptysetaS}
If $J_x\neq\mathbb{S}$ in \Cref{xisJalcove}, then $X_x(b)=\emptyset$.
\end{prop}
\begin{proof}
Note that $x$ is a $(J_x,v_xs_x)_\sigma$-alcove. Hence, it is enough to prove that\[\nu_{\lambda'}-\overline{\nu}_b\in Q_{J_x}^\vee\otimes \mathbb{Q}\]leads to a contradiction where $\lambda'\colonequals s_x \mu_x$ following the proof of \cite[Proposition 3.6.4]{GHN}. However, $\lambda'=s_x\mu_x$ is either $\mu_x-\alpha_x^\vee$ or $\mu_x$. The latter case is already proved\footnote{Note that, as mentioned before, the proof in \textit{loc.cit.} is incorrect when referring to `Section 3.5 (2)' and you need to use \Cref{wronglem} instead.} in \textit{loc.cit.}.

Now, let us consider the case $s_x\mu_x=\mu_x-\alpha_x^\vee$. We have \begin{equation}\label{qjx}
    \nu_{s_x\mu_x}-\bar{\nu}_b\in Q^\vee_{J_x,\mathbb{Q}}.
\end{equation}

Let $\varpi_i^\vee$ be the fundamental coweights and $\varpi_i$ be the fundamental weights. For simplicity, denote by $\mathcal{O}$ the $\sigma$-orbit of $s_x$ and by $\alpha_{\mathcal{O}}$ (resp. $\varpi_\calO$ or $\varpi_\calO^\vee$) the sum of the elements in the $\sigma$-orbit of $\alpha_x$ (resp. $\varpi_\calO$ or $\varpi_\calO^\vee$) and take the inner product of \eqref{qjx} with $\varpi_\calO$. Then we get $0$ on the right hand side and, on the left hand side, we have $\langle \varpi_\calO,s_x\mu_x\rangle = \langle s_x\varpi_\calO,\mu_x\rangle=\langle \varpi_\calO-\alpha_x,\mu_x\rangle=\langle\varpi_\calO,\mu_x\rangle-1$. 

Let $\mu_x$ be $\sum c_i\varpi_i^\vee$ for some $c_i\in\mathbb{Z}_{\ge0}$, we know that the coefficient $c_x$ of $\varpi_x^\vee$ is 1 from the assumption $s_x\mu_x = \mu_x -\alpha_x^\vee$. Hence, we have $1 = \langle \varpi_\mathcal{O} , \mu_x \rangle \ge \langle \varpi_x , \mu_x\rangle=\langle\varpi_x,\varpi_x^\vee+\sum_{i\neq x}c_i\varpi_i^\vee\rangle = \langle \varpi_x,\varpi_x^\vee\rangle + \sum_{i\neq x}c_i \langle \varpi_x, \varpi_i^\vee\rangle$. The first inequality comes from the fact that $\mu_x$ is dominant. In fact, if we denote the connected component of the Dynkin diagram containing $\alpha_x$ by $D$, the latter summation part is actually for $i$ such that $i\neq x$ and $\alpha_i$ in $D$. Note that $\langle\varpi_i,\varpi_j^\vee\rangle$ is the $(i,j)$-entry of the inverse of the Cartan matrix of $D$ whose entries are all positive. In most cases, this already exceeds $1$ by \cite[Reference Chapter. §2. Table 2]{OnishchikVinberg}. Note that $\langle \varpi_x,\varpi_x^\vee\rangle$ is the corresponding entry on the diagonal of the inverse matrix.

Therefore, there are two possible cases: a) $\langle\varpi_x,\varpi_x^\vee\rangle=1$ and all $c_i$'s are $0$ for $i$ such that $i\neq x$ and $\alpha_i$ in $D$ and b) $\langle\varpi_x,\varpi_x^\vee\rangle<1$. The first case is when $x$ is the unique vertex of degree=1 (with a single edge) in $B_\ge2$, $C_\ge2$, $D_\ge4$ or the middle vertex of $A_3=D_3$. The second case is when $x$ is one of two degree=1 vertices in $A_n$. As they contain a little tedious computation, we put it off to \Cref{postponedproof}.
\end{proof}

Combining \Cref{nonemptyetaS} and \Cref{nonemptysetaS}, we get
\[\label{lefttoright}
    X_x(b)\neq\emptyset\text{ only if }\sigmasupp(\eta_\sigma(x))=\sigmasupp(\sigma^{-1}(s_x)\eta_\sigma(x)s_x)=\mathbb{S}.\]




\begin{proof}[Proof of \Cref{mainone}]
Note again that $\Phi_x=\{\alpha_x\}$ so that $W_x=\{\id_{W_0},s_x\}$. Now, combine \Cref{righttoleftgen} and \Cref{lefttoright}.
\end{proof}

\subsection{On the condition $\widetilde{\sigmasupp}(x)=\widetilde{\bbS}$}
Finally, the following lemma explains why the assumption $\widetilde{\sigmasupp}(x)=\widetilde{\bbS}$ is not necessary in \Cref{mainthm} (1). Note that $\widetilde{\sigmasupp}(x)\neq\widetilde{\bbS}$ implies that $\nu_x$ is central.

\begin{lem}\label{nucentral}
    Let $x$ be in exactly one critical strip and suppose $\nu_x$ is central. Then both $\sigmasupp(\eta_\sigma(x))$ and $\sigmasupp(\sigma^{-1}(s_x)\eta_\sigma(x)s_x)$ are $\bbS$.
\end{lem}

\begin{proof}
    Note that $\nu_x=v_x\cdot \nu_{v_x^{-1}x\sigma(v_x)}$ and $v_x^{-1}x\sigma(v_x)=t^{\mu_x}w_x\sigma(v_x)=t^{\mu_x}\sigma(\eta_\sigma(x))$\footnote{Caution. This is not necessarily in the dominant Weyl chamber.}. Hence, $\nu_{t^{\mu_x} \sigma(\eta_\sigma(x))}$ is also central. On the other hand, similarly to \Cref{GHN351}, we have $\nu_{\mu_x}-\nu_{t^{\mu_x}\sigma(\eta_\sigma(x))}\in Q_{\sigmasupp(\sigma(\eta_\sigma(x))),\bbQ}^\vee$ by applying \Cref{wmu-mu}. As $\mu_x$ is dominant, $\nu_{\mu_x}$ is dominant. If $\mu_x$ is non-central, then $\nu_{\mu_x}$ is also non-central and so \Cref{wronglem} enforces $\sigmasupp(\sigma(\etasigma(x)))=\bbS$. Similarly as above, we have \[\nu_{s_x\cdot\mu_x}-\nu_{t^{s_x\cdot\mu_x}s_xw_x\sigma(v_x)\sigma(s_x)}\in Q_{\sigmasupp(s_x\sigma(\etasigma(x))\sigma(s_x)),\bbQ}^\vee.\] However, $t^{s_x\cdot\mu_x}s_xw_x\sigma(v_x)\sigma(s_x)= (v_xs_x)^{-1} x\sigma(v_xs_x)$, but $\nu_x$ is central so the second term is central. We can now repeat the proof of \Cref{nonemptysetaS}.
    
    When $\mu_x$ is central, we have $x=v_xt^{\mu_x}$. We can repeat \Cref{nonemptyetaS} to show that $\sigmasupp(\etasigma(x))=\bbS$. Moreover, the one critical strip assumption tells us that $w_0v_x^{-1}$ is a simple reflection $s$ (corresponding to $\alpha\in\Phi^+$). Then, $\alpha_x=v_x^{-1}\alpha$ and so $v_x\alpha_x\in\Phi^+$, that is, $\ell(v_xs_x)=\ell(v_x)+1=\ell(w_0)$. Hence, $v_x=w_0s_x$ and $\sigma^{-1}(s_x)\etasigma(x)s_x=\sigma^{-1}(s_x)w_0$ whose support is $\bbS$.
\end{proof}

\section{Some computations and applications}
In this section, we show computations for possibly exceptional cases (mentioned in \Cref{nonemptysetaS}) when $x$ belongs to exactly one critical strip. Moreover, we have some remarks on type $A_n$ case and we prove \Cref{bgx}.
\subsection{Completion of the proof for the one critical strip case}
\begin{proof}[Proof of \Cref{nonemptysetaS} (cont'd)]\label{postponedproof}We separate into two cases. (Note that we use \cite[Reference Chapter. \S2. Table 2]{OnishchikVinberg} here.)

\noindent\textit{Case 1-1}. $x$ is the vertex of degree $1$ (with a single edge) in $B_{\neq2},C_{\neq2},D_{\ge4}$.

$\langle\varpi_x,\varpi_x^\vee\rangle$ is already $1$. So, $c_i$'s are all zero, that is, $\mu_x=\varpi_x^\vee$. For simplicity, we use the notation $*_1$ instead of $*_x$ where $*=s,\alpha, \varpi^\vee$ in this paragraph. For any positive root $\alpha$ with no support at $\alpha_1$, we have $x\not\in CS_{v_x\alpha}$ and $\langle\alpha,\varpi_1^\vee\rangle=0$. So we have $v_x\alpha\in\Phi^-$ and $w_x^{-1}\alpha\in\Phi^+$ by \Cref{klem}. As $x\in CS_{v_x\alpha_1}$ and $\langle\alpha_1,\varpi_1^\vee\rangle=1$, we have $v_x\alpha_1\in\Phi^+$ and $w_x^{-1}\alpha_1\in\Phi^-$. Note that $s_2\alpha_1=\alpha_1+\alpha_2$ is different from $\alpha_1$ and in $\Phi^+$. Moreover, $\langle s_2\alpha_1,\varpi_1^\vee\rangle=1$. Due to the uniqueness assumption of critical strips, either $v_x(s_2\alpha_1)\in\Phi^-$ or $ w_x^{-1}(s_2\alpha_1)\in\Phi^+$ by \Cref{klem}.
\begin{enumerate}
    \item The latter case: $w_x^{-1}(\alpha_1+\alpha_2)\in\Phi^+$ so that $w_x^{-1}s_1\alpha\in\Phi^+$ for all positive simple roots $\alpha$. Hence, $w_x=s_1$ so that $\sigma^{-1}(s_1)\eta_\sigma(x)s_1= vs_1$. As $J_x=\bbS\setminus\calO$, $v_x\in W_{\bbS\setminus\calO}s_1$. However, $\Phi^+\ni v_x\alpha_1\in W_{\bbS\setminus\calO}(-\alpha_1)\subset \Phi^-$ as $1\in\calO$. Contradiction.
    \item The first case: Apply the same argument to $w_0v_x$ instead of $w_x^{-1}$ and get $v_x=w_0s_1$. Then $\sigma^{-1}(s_1)\etasigma(x)s_1=\sigma^{-1}(s_1w_x)w_0$. Consider $\sigma(w_0)w_x^{-1}s_1\alpha_1$ which is in $\Phi^-$ by above, but $\sigma(w_0)w_x^{-1}s_1\in W_{\bbS\setminus\calO}$ so it should be a positive root which is a contradiction.\vspace{-2pt}
\end{enumerate}

\noindent\textit{Case 1-2}. $x$ is the middle point in $A_3$.

$\mu_x=\varpi_2^\vee$. Similarly as before, $v_x\alpha_2\in\Phi^+$ and $w_x^{-1}\alpha_2\in\Phi^-$. It means that $w_x\in s_2W_{\{1,3\}}$ or $w_x=s_2s_1s_3s_2$. Hence, $s_2w_x\in W_{\{1,3\}}$ or $s_2w_x=s_1s_3s_2$ and $\sigma^{-1}(s_2)\etasigma(x)s_2=\sigma^{-1}(s_2w_x)v_xs_2$. As it belongs to $W_{\bbS\setminus \calO}$, we have $v_x\in W_{\bbS\setminus\calO}s_2$ so that $v_x\alpha_2\in\Phi^-$ which is a contradiction or $v_x\in s_2 W_{\bbS\setminus\calO}s_2$. In the latter case, note that $v_x$ is supported at $s_1$ because $v_x\alpha_1\in\Phi^-$ and so $v_x\alpha_2\in\Phi^-$ from $v_x\in s_2 W_{\bbS\setminus\calO}s_2$.\vspace{3pt}

\noindent\textit{Case 1-3}. $x$ is the vertex with an outward arrow in $B_2=C_2$.

May assume $B_2$ and $\mu_x=\varpi_1^\vee$. This is same as \textit{Case 1-1} as $s_2\alpha_1=\alpha_1+\alpha_2$.\vspace{3pt}

\noindent\textit{Case 2}. $x$ is a vertex of degree $1$ in $A_n$.

As the sum is $1$, $D\cap\calO=\{x\}$ necessarily and $\mu_x=\varpi_x^\vee+\varpi_y^\vee$ where $y$ is a vertex (possibly in another connected component) of degree $1$ not in $\calO$.

For simplicity, we will call $D'$ the connected component containing $\varpi_y^\vee$ and use $\varpi_{1,D}^\vee$ instead of $\varpi_x^\vee$ and $\varpi_{n,D'}^\vee$ instead of $\varpi_y^\vee$. Moreover, $\sigma$ does not shuffle $1$ and $n$ in (distinct) connected components.
\begin{enumerate}
    \item The case $D\neq D'$: The proof is exactly the same as \textit{Case 1-1}.
    \item The case $D=D'$: $\mu_x=\varpi_1^\vee+\varpi_n^\vee$ in the same connected component. The proof is similar. In a similar way, we have $v_x\alpha\in\Phi^-$ and $w_x^{-1}\alpha\in\Phi^+$ for $\alpha$ not supported at both $\alpha_1$ and $\alpha_n$. For $\alpha=\alpha_1$, we have the opposite situation. For the remaining case, it should not be the opposite situation.
    
    The method is essentially the same. Note that the goal in the above was to show that $w_x=s_1$ or $v_x=w_0s_1$. Here, the goal is to show that $w_x=s_1\cdots s_m$ for some $m\le n$ or $v_x=w_0s_m\cdots s_1$ for some $m\le n$. Using that the roots are of the form $\alpha_i+\alpha_{i+1}+\cdots+\alpha_j$, one can prove that. We skip the proof.\vspace{-16pt}
\end{enumerate}
\end{proof}

\subsection{Sharper bounds for \Cref{mainthm} (3)}

Denote by $Q^\vee_{>0}$ the set of positive $\bbQ$-linear sums of all simple coroots. Note that $\nu_{\mu_x}\in Q^\vee_{>0}$ unless $\mu_x$ is central.

Suppose that $r\in W_x$ and $s$ is a simple reflection such that $sr\in W_x$ and $sr>r$. Considering \Cref{rmu-mu}, we know that $sr\cdot\mu=r\cdot\mu- \langle r^{-1}\alpha,\mu\rangle\alpha^\vee$ where $\alpha$ is the simple positive root corresponding to $s$. Moreover, as $sr\in W_x$, $\langle r^{-1}\alpha,\mu\rangle=0$ or $1$. Hence, as long as $sr\in W_x$ (assuming $sr>r$), we know that $sr\cdot\mu$ is the same as $r\cdot\mu$ or $r\cdot\mu - \alpha^\vee$.

Therefore, if $\nu_{r'\cdot\mu_x}\in Q^\vee_{J,\bbQ}$ for some $J\subsetneq\bbS$ and $r'\in W_x$, we can find a minimal $r_0\in W_0$ such that $\ell(r'r_0^{-1})+\ell(r_0)=\ell(r')$ (so that $r_0\in W_x$ by \Cref{wxleftclosed}) and $r_0\cdot\mu\not\in Q^\vee_{>0}$ in the sense that $sr_0\cdot\mu\in Q^\vee_{>0}$ for any simple reflection $s$ such that $sr_0<r_0$. Note that it is not necessarily unique and we make any choice.

The minimality assumption tells us that any reduced expression $s_{i_k}\cdots s_{i_1}$ of $r_0$ has the same $s_{i_k}$ (the first simple reflection should be constant). Moreover, $\langle s_{i_1}\cdots s_{i_{k-1}}\alpha_{i_k},\mu_x\rangle=1$ considering \Cref{rmu-mu}. Finally, $\supp(r_0)$ should be connected. Now, the precise statement of \Cref{mainthm} (3) is the following.

\begin{prop}\label{Anresult}
    In type $A_n$, suppose that $\mu_x\in Q^\vee$ and assume $\langle\varpi_1,\mu_x\rangle>1$, $\langle\varpi_n,\mu_x\rangle>1$. Then \Cref{mainconj} holds.
\end{prop}

\begin{proof}
    Let $r_0$ be a minimal element chosen before and $m$ be the unique number such that $r_0\alpha_m<0$. Similar to \Cref{lefttorightgen}, it is enough to show that $r_0\cdot\mu_x\in Q^\vee_{\neq\bbS,\bbQ}$ is a contradiction. Note that $\mu_x\in Q^\vee$, so we can assume $r_0\cdot\mu_x\in Q_{\neq\bbS}^\vee$. Due to the minimality of $r_0$, we have that $\bbS\setminus\supp(r_0\cdot\mu_x)=\{m\}$. Moreover, for each $i\in\supp(r_0\cdot\mu_x)$, $\alpha_i^\vee$-coefficient of $r_0\cdot\mu_x$ is a positive integer. Therefore, if $m\neq1$ or $n$, we have $\langle\alpha_m,r_0\cdot\mu_x\rangle\le-2$. However, $r_0^{-1}\alpha_m\in-\Phi_x$ so that $\langle r_0^{-1}\alpha_m,\mu_x\rangle=-1$ which is a contradiction.
    
    When $m=1$ (resp. $n$), note that $r_0\cdot\mu_x\in Q_J^\vee$ (where $J=\bbS\setminus\{m\}$) is weakly dominant in the sense of \cite[Proposition 3.1]{Nie}. Therefore, $r_0\cdot\mu$ and $\mu$ fits into the setting in \cite[Lemma 5.9]{Nie} so that $\mu-r_0\cdot\mu$ is the sum of some positive roots orthogonal to each other. However, no two positive roots containing $\alpha_1$ (resp. $\alpha_n$) can be orthogonal. This implies that the $\alpha_1^\vee$(resp. $\alpha_n^\vee$)-coefficient of $\mu$ is either that of $r_0\cdot\mu$ (which is $0$) or one larger than that. So it must be $1$. This contradicts to the assumption.
\end{proof}

We remark that the above proof works for $D_n$ and $E_n$ types but need one more assumption that $\langle\varpi_{n-1},\mu_x\rangle>1$ where $n-1$ is the vertex of degree $1$ other than the vertices $1$ and $n$.

\subsection{Application: Cordial elements and generic $\sigma$-conjugacy class}
We summarize relevant concepts. For more details, we refer to \cite{MilVie}.

Given $x\in\extw$, let $B(G)_x$ be the set of $[b]\in B(G)$ such that $I\dot{x}I\cap [b]\neq\emptyset$. Then, $B(G)_x$ contains a unique maximal element called generic $\sigma$-conjugacy class and denoted by $[b_x]$.
\begin{defn}[Cordial element]\label{cordial}
    An element $x\in \extw$ is called cordial if\[\ell(x)-\ell(\etasigma(x))=\langle2\rho,\nu_x\rangle-\defe(b_x),\]where $2\rho$ is the sum of all positive coroots and $\defe(b_x)$ is the difference between the rank of $G$ and $J_{b_x}$ over $F$, the $\sigma$-centralizer of $b_x$ in $\brG$.
\end{defn}
For example, $x$ in the antidominant Weyl chamber is cordial (\textit{loc.cit.} Theorem 1.2).

\begin{lem}\label{saturated}
    Let $x$ be a cordial element. Then, $B(G)_x$ is saturated in the following sense:\begin{quotation}Suppose that $[b_1]$, $[b_2]\in\! B(G)_x$ satisfy $[b_1]\le[b_2]$. Then, for any $[b]\in\! B(G)$ such that $[b_1]\le[b]\le[b_2]$, we have $[b]\in\! B(G)_x$.\end{quotation}Here, $\le$ is the partial ordering defined in $B(G)$.
\end{lem}
Hence, if the minimal and maximal elements are known, we have the full description of $B(G)_x$ which implies the complete classification of the nonemptiness of $X_x(b)$ for a fixed $x$.

\begin{proof}[Proof of \Cref{bgx}]
By \cite[Proposition 4.2]{He21}, $x=vt^\mu$ for a dominant $\mu$ is cordial and the generic $\sigma$-conjugacy class is $[t^\mu]$. By \Cref{saturated}, we only need to describe the minimal element in $B(G)_x$.

For a central $\mu$, we note that $[t^\mu]$ is the minimal among the elements whose image under the Kottwitz map is $\kappa_G(t^\mu)$. As it is minimal and maximal at the same time, $B(G)_x=\{[t^\mu]\}=B(G,\mu)$.

Let $\mu$ be non-central. We proved $\Phi_x=\{\alpha\in\Phi^+:\langle\alpha,\mu\rangle=0\text{ and }v\alpha\in\Phi^+\}$ in the last part of \Cref{123proof}. Hence, $\mu$ is fixed by the reflection with respect to $H_\alpha$ for $\alpha\in \Phi_x$. Now, we get
\begin{align*}
W_x&=\{r=s_m\cdots s_1\in W_0: \alpha_1,s_1\alpha_2, \cdots,s_1\cdots s_{m-1}\alpha_m\in \Phi_x\}\\
&=\{r\in W_0:\supp(r)\subset W_0(\mu)\text{ and }v\alpha_1,\cdots,vs_1\cdots s_{m-1}\alpha_m\in\Phi^+\}\\&=\{r\in W_0(\mu):\ell(vr^{-1})=\ell(v)+\ell(r)\}
\end{align*}

Now, by \Cref{mainthm} (2), $X_x(b_\text{b})\neq\emptyset$ for the basic element $b_\text{b}$ satisfying $\kappa_G(x)=\kappa_G(b_\text{b})$ if $\sigmasupp(\sigma^{-1}(r)\etasigma(x) r^{-1})=\bbS$ for all $r\in W_x$. In such a case, by \Cref{saturated}, $B(G)_x=\{[b]\in B(G):[b_\text{b}]\le [b]\le [b_x]=[t^\mu]\}=\{[b]\in B(G):[b]\le[t^\mu]\}=B(G,\mu)$.
\end{proof}

\bibliographystyle{alpha}
\bibliography{bibibi}

@article {Ra00,
    AUTHOR = {Rapoport, Michael},
     TITLE = {A positivity property of the {S}atake isomorphism},
   JOURNAL = {Manuscripta Math.},
  FJOURNAL = {Manuscripta Mathematica},
    VOLUME = {101},
      YEAR = {2000},
    NUMBER = {2},
     PAGES = {153--166},
      ISSN = {0025-2611},
   MRCLASS = {22E50 (14F30 20G25)},
  MRNUMBER = {1742251},
MRREVIEWER = {Adolfo Quir\'{o}s},
       DOI = {10.1007/s002290050010},
       URL = {https://doi.org/10.1007/s002290050010},
}

@article {GHN19,
    AUTHOR = {G\"{o}rtz, Ulrich and He, Xuhua and Nie, Sian},
     TITLE = {Fully {H}odge-{N}ewton decomposable {S}himura varieties},
   JOURNAL = {Peking Math. J.},
  FJOURNAL = {Peking Mathematical Journal},
    VOLUME = {2},
      YEAR = {2019},
    NUMBER = {2},
     PAGES = {99--154},
      ISSN = {2096-6075},
   MRCLASS = {11G18 (14G35 20G25)},
  MRNUMBER = {4060001},
MRREVIEWER = {Jinbo Ren},
       DOI = {10.1007/s42543-019-00013-2},
       URL = {https://doi.org/10.1007/s42543-019-00013-2},
}

@article {Bea,
    AUTHOR = {Beazley, E. T.},
     TITLE = {Codimensions of {N}ewton strata for {${\rm SL}_3(F)$} in the
              {I}wahori case},
   JOURNAL = {Math. Z.},
  FJOURNAL = {Mathematische Zeitschrift},
    VOLUME = {263},
      YEAR = {2009},
    NUMBER = {3},
     PAGES = {499--540},
      ISSN = {0025-5874},
   MRCLASS = {20G25 (14L05)},
  MRNUMBER = {2545856},
MRREVIEWER = {Eva Viehmann},
       DOI = {10.1007/s00209-008-0429-z},
       URL = {https://doi.org/10.1007/s00209-008-0429-z},
}

@article {KR00,
    AUTHOR = {Kottwitz, R. and Rapoport, M.},
     TITLE = {Minuscule alcoves for {${\rm GL}_n$} and {$G{\rm Sp}_{2n}$}},
   JOURNAL = {Manuscripta Math.},
  FJOURNAL = {Manuscripta Mathematica},
    VOLUME = {102},
      YEAR = {2000},
    NUMBER = {4},
     PAGES = {403--428},
      ISSN = {0025-2611},
   MRCLASS = {20G15 (17B20)},
  MRNUMBER = {1785323},
MRREVIEWER = {Cheng Dong Chen},
       DOI = {10.1007/s002290070034},
       URL = {https://doi.org/10.1007/s002290070034},
}

@article {He21,
    AUTHOR = {He, Xuhua},
     TITLE = {Cordial elements and dimensions of affine {D}eligne-{L}usztig
              varieties},
   JOURNAL = {Forum Math. Pi},
  FJOURNAL = {Forum of Mathematics. Pi},
    VOLUME = {9},
      YEAR = {2021},
     PAGES = {Paper No. e9, 15},
   MRCLASS = {11G25 (20G25)},
  MRNUMBER = {4312326},
MRREVIEWER = {Olivier Dudas},
       DOI = {10.1017/fmp.2021.10},
       URL = {https://doi.org/10.1017/fmp.2021.10},
}

@article {MilVie,
    AUTHOR = {Mili\'{c}evi\'{c}, Elizabeth and Viehmann, Eva},
     TITLE = {Generic {N}ewton points and the {N}ewton poset in
              {I}wahori-double cosets},
   JOURNAL = {Forum Math. Sigma},
  FJOURNAL = {Forum of Mathematics. Sigma},
    VOLUME = {8},
      YEAR = {2020},
     PAGES = {Paper No. e50, 18},
   MRCLASS = {20G25 (11G25 14L30 20F55)},
  MRNUMBER = {4176754},
MRREVIEWER = {Zhe Chen},
       DOI = {10.1017/fms.2020.46},
       URL = {https://doi.org/10.1017/fms.2020.46},
}

@misc{Sch,
    title={Geometric Newton points and cordial elements},
    author={Schremmer, Felix},
    year={2022},
    eprint={2205.02039},
    archivePrefix={arXiv},
    primaryClass={math.RT}
}

@article {Win05,
    AUTHOR = {Wintenberger, J.-P.},
     TITLE = {Existence de {$F$}-cristaux avec structures suppl\'{e}mentaires},
   JOURNAL = {Adv. Math.},
  FJOURNAL = {Advances in Mathematics},
    VOLUME = {190},
      YEAR = {2005},
    NUMBER = {1},
     PAGES = {196--224},
      ISSN = {0001-8708},
   MRCLASS = {14F30 (20G25)},
  MRNUMBER = {2104909},
MRREVIEWER = {Fabrizio Andreatta},
       DOI = {10.1016/j.aim.2003.12.006},
       URL = {https://doi.org/10.1016/j.aim.2003.12.006},
}

@article {KoRa03,
    AUTHOR = {Kottwitz, R. and Rapoport, M.},
     TITLE = {On the existence of {$F$}-crystals},
   JOURNAL = {Comment. Math. Helv.},
  FJOURNAL = {Commentarii Mathematici Helvetici},
    VOLUME = {78},
      YEAR = {2003},
    NUMBER = {1},
     PAGES = {153--184},
      ISSN = {0010-2571},
   MRCLASS = {11S25 (14F30 14L05)},
  MRNUMBER = {1966756},
MRREVIEWER = {Elmar Grosse-Kl\"{o}nne},
       DOI = {10.1007/s000140300007},
       URL = {https://doi.org/10.1007/s000140300007},
}

@article {He16,
    AUTHOR = {He, Xuhua},
     TITLE = {Kottwitz-{R}apoport conjecture on unions of affine
              {D}eligne-{L}usztig varieties},
   JOURNAL = {Ann. Sci. \'{E}c. Norm. Sup\'{e}r. (4)},
  FJOURNAL = {Annales Scientifiques de l'\'{E}cole Normale Sup\'{e}rieure. Quatri\`eme
              S\'{e}rie},
    VOLUME = {49},
      YEAR = {2016},
    NUMBER = {5},
     PAGES = {1125--1141},
      ISSN = {0012-9593},
   MRCLASS = {14M15 (14G35 20G25)},
  MRNUMBER = {3581812},
MRREVIEWER = {Alexander Boris Ivanov},
       DOI = {10.24033/asens.2305},
       URL = {https://doi.org/10.24033/asens.2305},
}

@article {Kottwitz03,
    AUTHOR = {Kottwitz, Robert E.},
     TITLE = {On the {H}odge-{N}ewton decomposition for split groups},
   JOURNAL = {Int. Math. Res. Not.},
  FJOURNAL = {International Mathematics Research Notices},
      YEAR = {2003},
    NUMBER = {26},
     PAGES = {1433--1447},
      ISSN = {1073-7928},
   MRCLASS = {22E50 (20G25)},
  MRNUMBER = {1976046},
MRREVIEWER = {Abdellah Mokrane},
       DOI = {10.1155/S1073792803130218},
       URL = {https://doi.org/10.1155/S1073792803130218},
}

@article {He14,
    AUTHOR = {He, Xuhua},
     TITLE = {Geometric and homological properties of affine
              {D}eligne-{L}usztig varieties},
   JOURNAL = {Ann. of Math. (2)},
  FJOURNAL = {Annals of Mathematics. Second Series},
    VOLUME = {179},
      YEAR = {2014},
    NUMBER = {1},
     PAGES = {367--404},
      ISSN = {0003-486X},
   MRCLASS = {14L35 (14M15 20G25)},
  MRNUMBER = {3126571},
MRREVIEWER = {Alexander Boris Ivanov},
       DOI = {10.4007/annals.2014.179.1.6},
       URL = {https://doi.org/10.4007/annals.2014.179.1.6},
}

@article {Gashi,
    AUTHOR = {Gashi, Q\"{e}ndrim R.},
     TITLE = {On a conjecture of {K}ottwitz and {R}apoport},
   JOURNAL = {Ann. Sci. \'{E}c. Norm. Sup\'{e}r. (4)},
  FJOURNAL = {Annales Scientifiques de l'\'{E}cole Normale Sup\'{e}rieure. Quatri\`eme
              S\'{e}rie},
    VOLUME = {43},
      YEAR = {2010},
    NUMBER = {6},
     PAGES = {1017--1038},
      ISSN = {0012-9593},
   MRCLASS = {14L35 (20G15)},
  MRNUMBER = {2778454},
       DOI = {10.24033/asens.2138},
       URL = {https://doi.org/10.24033/asens.2138},
}

@article {Kis17,
    AUTHOR = {Kisin, Mark},
     TITLE = {{${\rm mod}\,p$} points on {S}himura varieties of abelian
              type},
   JOURNAL = {J. Amer. Math. Soc.},
  FJOURNAL = {Journal of the American Mathematical Society},
    VOLUME = {30},
      YEAR = {2017},
    NUMBER = {3},
     PAGES = {819--914},
      ISSN = {0894-0347},
   MRCLASS = {11G18 (11G10 14G35)},
  MRNUMBER = {3630089},
MRREVIEWER = {Mihran Papikian},
       DOI = {10.1090/jams/867},
       URL = {https://doi.org/10.1090/jams/867},
}

@article {Shen,
    AUTHOR = {Shen, Xu},
     TITLE = {On some generalized {R}apoport-{Z}ink spaces},
   JOURNAL = {Canad. J. Math.},
  FJOURNAL = {Canadian Journal of Mathematics. Journal Canadien de
              Math\'{e}matiques},
    VOLUME = {72},
      YEAR = {2020},
    NUMBER = {5},
     PAGES = {1111--1187},
      ISSN = {0008-414X},
   MRCLASS = {11G18 (14G35)},
  MRNUMBER = {4152538},
MRREVIEWER = {Alan Koch},
       DOI = {10.4153/s0008414x19000269},
       URL = {https://doi.org/10.4153/s0008414x19000269},
}

@article {CKV,
    AUTHOR = {Chen, Miaofen and Kisin, Mark and Viehmann, Eva},
     TITLE = {Connected components of affine {D}eligne-{L}usztig varieties
              in mixed characteristic},
   JOURNAL = {Compos. Math.},
  FJOURNAL = {Compositio Mathematica},
    VOLUME = {151},
      YEAR = {2015},
    NUMBER = {9},
     PAGES = {1697--1762},
      ISSN = {0010-437X},
   MRCLASS = {14G35 (20G25)},
  MRNUMBER = {3406443},
MRREVIEWER = {Zongbin Chen},
       DOI = {10.1112/S0010437X15007253},
       URL = {https://doi.org/10.1112/S0010437X15007253},
}

@article {Nie,
    AUTHOR = {Nie, Sian},
     TITLE = {Connected components of closed affine {D}eligne-{L}usztig
              varieties in affine {G}rassmannians},
   JOURNAL = {Amer. J. Math.},
  FJOURNAL = {American Journal of Mathematics},
    VOLUME = {140},
      YEAR = {2018},
    NUMBER = {5},
     PAGES = {1357--1397},
      ISSN = {0002-9327},
   MRCLASS = {14M15 (14G35)},
  MRNUMBER = {3862068},
MRREVIEWER = {\'{O}scar Rivero},
       DOI = {10.1353/ajm.2018.0034},
       URL = {https://doi.org/10.1353/ajm.2018.0034},
}

@article {Djokovic,
    AUTHOR = {{\DJ}oković, D. \v{Z}. and Check, P. and H\'{e}e, J.-Y.},
     TITLE = {On closed subsets of root systems},
   JOURNAL = {Canad. Math. Bull.},
  FJOURNAL = {Canadian Mathematical Bulletin. Bulletin Canadien de
              Math\'{e}matiques},
    VOLUME = {37},
      YEAR = {1994},
    NUMBER = {3},
     PAGES = {338--345},
      ISSN = {0008-4395},
   MRCLASS = {17B20},
  MRNUMBER = {1289769},
MRREVIEWER = {Jorge A. Vargas},
       DOI = {10.4153/CMB-1994-050-4},
       URL = {https://doi.org/10.4153/CMB-1994-050-4},
}

@book {OnishchikVinberg,
    AUTHOR = {Onishchik, A. L. and Vinberg, \`E. B.},
     TITLE = {Lie groups and algebraic groups},
    SERIES = {Springer Series in Soviet Mathematics},
      NOTE = {Translated from the Russian and with a preface by D. A.
              Leites},
 PUBLISHER = {Springer-Verlag, Berlin},
      YEAR = {1990},
     PAGES = {xx+328},
      ISBN = {3-540-50614-4},
   MRCLASS = {22-01 (17B20 20G20 22E10 22E15)},
  MRNUMBER = {1064110},
MRREVIEWER = {James E. Humphreys},
       DOI = {10.1007/978-3-642-74334-4},
       URL = {https://doi.org/10.1007/978-3-642-74334-4},
}

@article {LusztigTits,
    AUTHOR = {Lusztig, George and Tits, Jacques},
     TITLE = {The inverse of a {C}artan matrix},
   JOURNAL = {An. Univ. Timi\c{s}oara Ser. \c{S}tiin\c{t}. Mat.},
  FJOURNAL = {Analele Universit\u{a}\c{t}ii din Timi\c{s}oara. Seria \c{S}tiin\c{t}e Matematice},
    VOLUME = {30},
      YEAR = {1992},
    NUMBER = {1},
     PAGES = {17--23},
   MRCLASS = {22E15 (15A09)},
  MRNUMBER = {1329156},
MRREVIEWER = {Jean H. Bevis},
}

@article{ghnerror,
        title={ERRATUM: {$\bf P$}-alcoves and nonemptiness of affine {D}eligne-{L}usztig varieties},
        author={Görtz, Ulrich and He, Xuhua and Nie, Sian},
        journal={https://www.uni-due.de/~hx0050/pdf/Erratum-GHN.pdf}
}

@article {HeZhou,
    AUTHOR = {He, Xuhua and Zhou, Rong},
     TITLE = {On the connected components of affine {D}eligne-{L}usztig
              varieties},
   JOURNAL = {Duke Math. J.},
  FJOURNAL = {Duke Mathematical Journal},
    VOLUME = {169},
      YEAR = {2020},
    NUMBER = {14},
     PAGES = {2697--2765},
      ISSN = {0012-7094},
   MRCLASS = {14G35 (20G25)},
  MRNUMBER = {4149507},
MRREVIEWER = {Alan Koch},
       DOI = {10.1215/00127094-2020-0020},
       URL = {https://doi.org/10.1215/00127094-2020-0020},
}

@article {HeNie,
    AUTHOR = {He, Xuhua and Nie, Sian},
     TITLE = {On the acceptable elements},
   JOURNAL = {Int. Math. Res. Not. IMRN},
  FJOURNAL = {International Mathematics Research Notices. IMRN},
      YEAR = {2018},
    NUMBER = {3},
     PAGES = {907--931},
      ISSN = {1073-7928},
   MRCLASS = {14L15 (11G05 20G15)},
  MRNUMBER = {3801450},
MRREVIEWER = {Alan Koch},
       DOI = {10.1093/imrn/rnw260},
       URL = {https://doi.org/10.1093/imrn/rnw260},
}

@article {RR96,
    AUTHOR = {Rapoport, M. and Richartz, M.},
     TITLE = {On the classification and specialization of {$F$}-isocrystals
              with additional structure},
   JOURNAL = {Compositio Math.},
  FJOURNAL = {Compositio Mathematica},
    VOLUME = {103},
      YEAR = {1996},
    NUMBER = {2},
     PAGES = {153--181},
      ISSN = {0010-437X},
   MRCLASS = {14F30 (22E50)},
  MRNUMBER = {1411570},
MRREVIEWER = {Abdellah Mokrane},
       URL = {http://www.numdam.org/item?id=CM_1996__103_2_153_0},
}

@article {GHKR,
    AUTHOR = {G\"{o}rtz, Ulrich and Haines, Thomas J. and Kottwitz, Robert E.
              and Reuman, Daniel C.},
     TITLE = {Affine {D}eligne-{L}usztig varieties in affine flag varieties},
   JOURNAL = {Compos. Math.},
  FJOURNAL = {Compositio Mathematica},
    VOLUME = {146},
      YEAR = {2010},
    NUMBER = {5},
     PAGES = {1339--1382},
      ISSN = {0010-437X},
   MRCLASS = {14L35 (11S25 14F30 20G25)},
  MRNUMBER = {2684303},
MRREVIEWER = {Alan Koch},
       DOI = {10.1112/S0010437X10004823},
       URL = {https://doi.org/10.1112/S0010437X10004823},
}

@article {GHN,
    AUTHOR = {G\"{o}rtz, Ulrich and He, Xuhua and Nie, Sian},
     TITLE = {{$\bf P$}-alcoves and nonemptiness of affine
              {D}eligne-{L}usztig varieties},
   JOURNAL = {Ann. Sci. \'{E}c. Norm. Sup\'{e}r. (4)},
  FJOURNAL = {Annales Scientifiques de l'\'{E}cole Normale Sup\'{e}rieure. Quatri\`eme
              S\'{e}rie},
    VOLUME = {48},
      YEAR = {2015},
    NUMBER = {3},
     PAGES = {647--665},
      ISSN = {0012-9593},
   MRCLASS = {14L35 (14M15 20G15)},
  MRNUMBER = {3377055},
MRREVIEWER = {Eva Viehmann},
       DOI = {10.24033/asens.2254},
       URL = {https://doi.org/10.24033/asens.2254},
}

@article{HainesRapoport,
title = {Appendix: On parahoric subgroups},
journal = {Advances in Mathematics},
volume = {219},
number = {1},
pages = {188-198},
year = {2008},
issn = {0001-8708},
doi = {https://doi.org/10.1016/j.aim.2008.04.020},
url = {https://www.sciencedirect.com/science/article/pii/S0001870808001333},
author = {T. Haines and M. Rapoport}
}

@article{PappasRapoport,
title = {Twisted loop groups and their affine flag varieties},
journal = {Advances in Mathematics},
volume = {219},
number = {1},
pages = {118-198},
year = {2008},
issn = {0001-8708},
doi = {https://doi.org/10.1016/j.aim.2008.04.006},
url = {https://www.sciencedirect.com/science/article/pii/S0001870808001205},
author = {G. Pappas and M. Rapoport}
}

@incollection {Ra05,
    AUTHOR = {Rapoport, Michael},
     TITLE = {A guide to the reduction modulo {$p$} of {S}himura varieties},
      NOTE = {Automorphic forms. I},
   JOURNAL = {Ast\'{e}risque},
  FJOURNAL = {Ast\'{e}risque},
    NUMBER = {298},
      YEAR = {2005},
     PAGES = {271--318},
      ISSN = {0303-1179},
   MRCLASS = {11G18 (11G40 14G35)},
  MRNUMBER = {2141705},
MRREVIEWER = {Ulrich G\"{o}rtz},
}

@article {Reuman,
    AUTHOR = {Reuman, Daniel C.},
     TITLE = {Formulas for the dimensions of some affine {D}eligne-{L}usztig
              varieties},
   JOURNAL = {Michigan Math. J.},
  FJOURNAL = {Michigan Mathematical Journal},
    VOLUME = {52},
      YEAR = {2004},
    NUMBER = {2},
     PAGES = {435--451},
      ISSN = {0026-2285},
   MRCLASS = {20G25},
  MRNUMBER = {2069809},
MRREVIEWER = {Elmar Grosse-Kl\"{o}nne},
       DOI = {10.1307/mmj/1091112084},
       URL = {https://doi.org/10.1307/mmj/1091112084},
}

@book {Bourbaki,
    AUTHOR = {Bourbaki, Nicolas},
     TITLE = {\'{E}l\'{e}ments de math\'{e}matique},
      NOTE = {Groupes et alg\`ebres de Lie. Chapitres 4, 5 et 6. [Lie groups
              and Lie algebras. Chapters 4, 5 and 6]},
 PUBLISHER = {Masson, Paris},
      YEAR = {1981},
     PAGES = {290},
      ISBN = {2-225-76076-4},
   MRCLASS = {17-02 (00A05)},
  MRNUMBER = {647314},
}

@article {BS,
    AUTHOR = {Bhatt, Bhargav and Scholze, Peter},
     TITLE = {Projectivity of the {W}itt vector affine {G}rassmannian},
   JOURNAL = {Invent. Math.},
  FJOURNAL = {Inventiones Mathematicae},
    VOLUME = {209},
      YEAR = {2017},
    NUMBER = {2},
     PAGES = {329--423},
      ISSN = {0020-9910},
   MRCLASS = {14F05 (14M15 19G12)},
  MRNUMBER = {3674218},
MRREVIEWER = {Marc-Hubert Nicole},
       DOI = {10.1007/s00222-016-0710-4},
       URL = {https://doi.org/10.1007/s00222-016-0710-4},
}

@article {Zhu,
    AUTHOR = {Zhu, Xinwen},
     TITLE = {Affine {G}rassmannians and the geometric {S}atake in mixed
              characteristic},
   JOURNAL = {Ann. of Math. (2)},
  FJOURNAL = {Annals of Mathematics. Second Series},
    VOLUME = {185},
      YEAR = {2017},
    NUMBER = {2},
     PAGES = {403--492},
      ISSN = {0003-486X},
   MRCLASS = {14D24 (14L35 14M15 20G25)},
  MRNUMBER = {3612002},
MRREVIEWER = {Rolf Berndt},
       DOI = {10.4007/annals.2017.185.2.2},
       URL = {https://doi.org/10.4007/annals.2017.185.2.2},
}

@article {Kottwitz,
    AUTHOR = {Kottwitz, Robert E.},
     TITLE = {Isocrystals with additional structure. {II}},
   JOURNAL = {Compositio Math.},
  FJOURNAL = {Compositio Mathematica},
    VOLUME = {109},
      YEAR = {1997},
    NUMBER = {3},
     PAGES = {255--339},
      ISSN = {0010-437X},
   MRCLASS = {20G25 (11S25 14F30 14L05)},
  MRNUMBER = {1485921},
MRREVIEWER = {Guy Rousseau},
       DOI = {10.1023/A:1000102604688},
       URL = {https://doi.org/10.1023/A:1000102604688},
}

@article {Sopkina,
    AUTHOR = {Sopkina, E. A.},
     TITLE = {On the sum of roots of a closed set},
   JOURNAL = {Zap. Nauchn. Sem. S.-Peterburg. Otdel. Mat. Inst. Steklov.
              (POMI)},
  FJOURNAL = {Rossi\u{\i}skaya Akademiya Nauk. Sankt-Peterburgskoe Otdelenie.
              Matematicheski\u{\i} Institut im. V. A. Steklova. Zapiski Nauchnykh
              Seminarov (POMI)},
    VOLUME = {289},
      YEAR = {2002},
    NUMBER = {Vopr. Teor. Predst. Algebr. i Grupp. 9},
     PAGES = {277--286, 304},
      ISSN = {0373-2703},
   MRCLASS = {17B20},
  MRNUMBER = {1949746},
MRREVIEWER = {Alexander N. Rudy\u{\i}},
       DOI = {10.1023/B:JOTH.0000042319.81176.99},
       URL = {https://doi.org/10.1023/B:JOTH.0000042319.81176.99},
}

\end{document}